\documentclass[oneside,reqno]{amsart}

\def\BibTeX{{\rm B\kern-.05em{\sc i\kern-.025em b}\kern-.08em
    T\kern-.1667em\lower.7ex\hbox{E}\kern-.125emX}}

\usepackage{amssymb}

\usepackage{epsfig}  		

\usepackage{hyperref}
\usepackage{breakurl}

\newtheorem{theorem}{Theorem}[section]
\newtheorem{lemma}[theorem]{Lemma}
\newtheorem{proposition}[theorem]{Proposition}
\newtheorem{corollary}[theorem]{Corollary}

\theoremstyle{remark}
\theoremstyle{definition}

\newtheorem{definition}{Definition}[section]

\def\subsection{\@startsection{subsection}{2}%
  \z@{.5\linespacing\@plus.7\linespacing}{-.5em}%
  {\normalfont\bfseries}}

\newcommand{\mres}{\mathbin{\vrule height 1.1ex depth 0pt width
0.13ex\vrule height 0.13ex depth 0pt width 1.1ex}}

\makeatletter
\renewcommand\subsection{\@startsection{subsection}{2}%
  \z@{-.5\linespacing\@plus-.7\linespacing}{.5\linespacing}%
  {\normalfont\scshape}}
\renewcommand\subsubsection{\@startsection{subsubsection}{3}%
  \z@{.5\linespacing\@plus.7\linespacing}{-.5em}%
  {\normalfont\scshape}}
\makeatother

\DeclareSymbolFont{yhlargesymbols}{OMX}{yhex}{m}{n}
\DeclareMathAccent{\wideparen}{\mathord}{yhlargesymbols}{"F3}

\begin{document}


\title[The first \textit{\small \lowercase{p}}-widths of the unit disk]{The first {\normalfont \textit{\Large \lowercase{p}}}-widths of the unit disk}


\author{Sidney Donato}
\address{Instituto de Matem\'atica, Universidade Federal de Alagoas, 
Macei\'o, AL, Brazil.}
\email{sidney.silva@im.ufal.br}
\thanks{Research supported by the Coordena\c{c}\~{a}o de Aperfei\c{c}oamento de Pessoal de N\'{i}vel Superior (CAPES) - Brazil (CAPES) - Finance Code 001.}
\thanks{\textit{Mathematics Subject Classication} 2020. Primary: 53C23; Secondary: 58E05.}





\begin{abstract}
On compact 2-manifolds with non-empty strictly convex \-{boun}\-{da}\-{ry}, we prove a regularity result for integral 1-varifolds $V$ that are \-{sta}\-{tio}\-{na}\-{ry} with free boundary and $\mathbb{Z}_2$-almost minimizing in small annuli. That regularity says that $V$ is a free boundary finite geodesic network. Next, using that regularity, we compute the first $p$-widths of the unit closed ball $B^2,$ for $p=1, . . . , 4.$\end{abstract}

 \maketitle





\section{Introduction}
\label{intro}

For $n>0,$ let $(M^{n+1}, g)$ be a compact Riemannian manifold with (possibly empty) \-{boun}\-{da}\-{ry.} Gromov  \cite{Gromov-homotopy,Gromov-width}, Guth \cite{Guth}, Marques and Neves \cite{Marques-Neves} introduced the notion of volume spectrum $\{\omega_p(M)\}_{p=1}^{\infty}$ for the area functional in the space of relative mod 2 cycles $\mathcal{Z}_n(M, \partial M, \mathbb{Z}_2).$ The volume spectrum is a sequence of positive numbers that satisfies similar properties to the spectrum of the Laplacian: 
$$0<\omega_p(M)\leq \omega_{p+1} (M) \quad  \mbox{and} \quad  \omega_p(M) \rightarrow \infty \ \mbox{as} \ p \rightarrow \infty.$$
Where $\omega_p(M)$ is called the $p^{th}$ \textit{min-max width} of $M.$

The Liokumovich-Marques-Neves-Weyl \cite{Marques-Neves-Liokumovich} law gives the asymptotic behavior of this spectrum, precisely:
$$\lim_{p\rightarrow \infty} \omega_p (M) p^{-\frac{1}{n+1}}=a(n) \mbox{vol}(M)^{\frac{n}{n+1}},$$
In contrast with the Weyl law for eigenvalues of the Laplacian, the constant above is almost totally unknown. Obviously, if we have a full description of the values of the volume spectrum, we can deduce the constant $a(n).$ So far this full description seems to be very hard. In fact, the results on this direction got only to compute some initial widths (Aiex \cite{Aiex}, Gaspar and Guaraco \cite{Gaspar-Guaraco}, and Nurser \cite{Nurser}). More recently, Chodosh-Mantoulidis \cite{Chodosh-Mantoulidis} gives a description of this spectrum for the round two-sphere, and they found that $a(n)=\sqrt{\pi}$ in this case.

The main objective of this article is to compute the first widths of the unit closed ball (unit disk) $B^2 \subset \mathbb{R}^2$ and of planar full ellipses $E^2$ closed to $B^2$ (see Theorem \ref{theo.main.widths}).


The ideas to prove our results are similar to what was done by Aiex \cite{Aiex} for the 2-sphere $S^2$ and for ellipsoids close to $S^2.$ In this case, it is used a regularity result due to Allard and Almgren \cite{Allard-Almgren} which says that stationary integral 1-varifolds on closed Riemannian manifolds are finite geodesic networks. This means that the varifold is a finite union of geodesic segments such that the singularities are given by their possible stationary junctions. This regularity is an important tool, because it holds for the 1-varifolds obtained in the Min-Max Theorem, so for each $p \in \mathbb{N}$ there exist geodesic networks sufficiently close to achieve the $p$-width. In our case, we had to extend this regularity result for two dimensional manifolds $M^2$ with non-empty boundary and we did this supposing the boundary strictly convex. In this hypothesis we get (see Theorem \ref{theo.main.regularity}): \textit{If $V$ is a stationary integral 1-varifold which is $\mathbb{Z}_2$-almost \-{mi}\-{ni}\-{mi}\-{zing} in small annuli, then $V$ is a free boundary finite geodesic network.} This means that $V$ restricted to the interior of $M,$ $\mbox{int}(M),$ is a finite geodesic network, each geodesic segment has its interior in $\mbox{int}(M),$ and each point $p \in \partial M$ that is on the support of $V$ is given by the intersection of \-{boun}\-{da}\-{ries} of geodesic segments from $\mbox{int}(M)$ such that: if each of these segments at $p$ are \-{pa}\-{ra}\-{me}\-{te}\-{ri}\-{zed} to start at $p,$ then the resultant of the unit tangent vectors of the segments (and its multiplicities) is perpendicular to $\partial M$ at $p.$

The extra hypothesis that $V$ is $\mathbb{Z}_2$-almost minimizing in small annuli is a classical hypothesis to get regularity for the codimension one case and for $3\leq n+1\leq 7$ (see Pitts \cite{Pitts}, Simon \cite{LeonSimon} and Li and Zhou \cite{Li-Zhou}). Essentially, the regularity comes from the fact that almost minimizing varifolds are locally stable almost everywhere. The hypothesis of strictly convex boundary follows the ideas from \cite{Li-Zhou}, where they prove a regularity result for strictly convex boundary and $3\leq n+1 \leq 7.$ 

As in \cite{Aiex}, the regularity is an important step to calculate the $p$-widths. In fact, by that regularity the varifolds obtained in our adapted version of the Min-Max Theorem (see Theorem \ref{minmax.theorem}) are free boundary finite geodesic networks. We did a classification (Theorems \ref{classif.geod.net.disc} and  \ref{classif.geod.net.ellip}) of these varifolds which have low mass in $B^2$ and $E^2$ and then, we get candidates for the first $p$-widths. 

Finally, to compute the $p$-widths of $B^2$ we use $p$-sweepouts whose image are given by real algebraic varieties restricted to $B^2.$ We estimate these $p$-sweepouts and we combine with the classification to deduce the first widths. For $E^2$ we do similarly and using continuity.

This article is organized in the following way: in Section \ref{sec:2} we remember some basic theory and we give some definitions, also we explain how we adapt the Min-Max Theorem for our case (Theorem \ref{minmax.theorem} and Corollary \ref{approx.width}); in Section \ref{sec:3} we talk about free boundary geodesic networks and its properties, also we classify the free boundary finite geodesic networks which are $\mathbb{Z}_2$-almost \-{mi}\-{ni}\-{mi}\-{zing} in small annuli and have low mass in $B^2$ and $E^2,$ and we conclude proving our regularity result (Theorem \ref{theo.main.regularity}); in the Section \ref{sec:4} we compute the first $p$-widths of $B^2$ and $E^2$ (Theorem \ref{theo.main.widths}) using the regularity, classification and the estimates obtained for the $p$-sweepouts; and in the Section \ref{appendix} (Appendix) we prove the sharp estimate for these $p$-sweepouts.

\section{Preliminaries}
\label{sec:2}

Throughout this section $M$ denotes a compact Riemannian $(n+1)$-manifold, $n\geq 0,$ with smooth and possibly empty boundary $\partial M.$ We can always assume that $M$ is isometrically embedded in some Euclidean space $\mathbb{R}^Q$ for some $Q \in \mathbb{Z}_{+}.$ We denote by $B_r (p)$ as the open Euclidean ball of radius $r$ centered at $p\in \mathbb{R}^Q,$ and $A_{s, r}(p)$ the open annulus $B_r(p) \backslash \overline{{B}}_s (p)$ for $0<s<r.$  

When $M$ has non-empty boundary, the embedding above is obtained in the following way: we can extend $M$ to a closed Riemannian manifold $\widetilde{M}$ with the same dimension such that $M \subset \widetilde{M}$ (see Pigola-Veronelli \cite{Pigola-Veronelli}), and so, by the Nash's Theorem, we get the isometric embedding $\widetilde{M}\hookrightarrow \mathbb{R}^Q.$ We denote by $\widetilde{\mathcal{B}}_r(p)$ as the open geodesic ball in $\widetilde{M}$ of radius $r$ centered at $p.$ 

We consider the following spaces of vector fields:
$$\mathfrak{X} (M):=\{X \in \mathfrak{X} (\mathbb{R}^Q): X(p)\in T_p M \ \mbox{for all} \ p \in M\}   $$
and
$$\mathfrak{X}_{tan} (M):=\{X \in \mathfrak{X} (M): X(p)\in T_p (\partial M) \ \mbox{for all} \ p \in \partial M\}.$$

\begin{definition} 
	(Relative Topology) Given any subset $A \subset M,$ where $A$ is equipped with the subspace topology, the \textit{interior relative of }$A,$ $\mbox{int}_M (A),$ is defined as the set of all $p \in M$ such that there exists a relatively open neighborhood $U \subset A$ of $p.$ The \textit{exterior relative of} $A$ is denoted by $\mbox{int}_M (M\backslash A).$ And the \textit{relative boundary of} $A,$ $\partial_{rel} A,$ is the subset of $M$ such that is neither in the relative interior nor exterior of $A.$
\end{definition}

\begin{definition} 
	(Relative Convexity) A subset $\Omega \subset M$ is said to be a \textit{relatively convex} (respect. \textit{relatively strictly convex})  \textit{domain in} $M$ if it is a relatively open connected subset in $M$ whose relative boundary $\partial_{rel} \Omega$ is smooth and convex (respect. strictly convex) in $M.$
\end{definition}

\begin{definition}  (Fermi coordinates) 
	Given $p \in \partial M$ and suppose that the coordinates $(x_1, \cdots, x_n)$ are the geodesic normal coordinates of $\partial M$ in a neighborhood of $p.$ Take $t=\mbox{dist}_M( \ \cdot \ , \partial M),$ which is a smooth map well-defined in a relatively open neighborhood of $p$ in $M.$ \textit{The Fermi coordinates system of} $(M, \partial M)$ \textit{centered at} $p$ is given by the coordinates $(x_1, \cdots, x_n, t).$ Also, the \textit{Fermi distance function from} $p$ on a relatively open neighborhood of $p$ in $M$ is defined by
	$$\widetilde{r}:=\widetilde{r}_p(q)=|(x,t)|=\sqrt{x^2_1+\cdots + x^2_n+t^2}.$$
	
\end{definition}

\begin{definition} 
	Given $p \in \partial M,$ we define the \textit{Fermi half-ball and half-sphere of radius} $r$ \textit{centered at} $p$ respectively by
	$$\widetilde{\mathcal{B}}^{+}_{r}(p):=\{q \in M: \widetilde{r}_p (q)<r \}, \quad \widetilde{\mathcal{S}}^{+}_{r}(p):=\{q \in M: \widetilde{r}_p (q)=r \}.$$
	
\end{definition}

Also we consider the following open annular neighborhood in the Fermi coordinates:
$$\mathcal{A}_{s, t}(p):=\widetilde{\mathcal{B}}^+_t (p) \backslash \mbox{Clos}(\widetilde{\mathcal{B}}^+_s (p))$$
for $p \in \partial M,$ and $0<s<t.$ Where $\mbox{Clos}(\widetilde{\mathcal{B}}^+_s (p))$ denotes the closure of $\widetilde{\mathcal{B}}^+_s (p)$ on $M.$ Also, when $p \in \mbox{int}(M),$ we require that $t<\mbox{dist}_M(p, \partial M).$

The geometric properties of the Fermi half-ball and half-sphere can be summarized  in the following proposition:

\begin{proposition} {\normalfont \cite[Lemma A.5]{Li-Zhou} }\label{Fermi.Convex.theo}
	There exists a small constant $r_{Fermi}>0,$ depending only on the isometric embedding $M \subset \mathbb{R}^Q,$ such that for all $0< r < r_{Fermi}$
\begin{itemize}
  \item[(i)] $\widetilde{\mathcal{S}}^{+}_{r}(p)$ is a smooth hypersurface meeting $\partial M$ orthogonally;
  \item[(ii)] $\widetilde{\mathcal{B}}^{+}_{r}(p)$ is a relatively strictly convex \footnote{The convexity in \cite{Li-Zhou} is assumed to be strict convexity.} domain in $M;$
  \item[(iii)] $B_{r/2}(p) \cap M \subset \widetilde{\mathcal{B}}^{+}_{r}(p) \subset B_{2r}(p) \cap M.$
\end{itemize}

\end{proposition}

\subsection{Relative Flat Cycles}
\label{section2.1}

We recall some definitions that can be found in \cite[Section 2]{Marques-Neves-Morse-index} or in \cite[Section 4]{Federer}. For each $0\leq k \leq n+1,$ $\mathcal{R}_k(M; \mathbb{Z}_2)$ denotes the set of $k$\textit{-dimensional rectifiable mod 2 flat chains} in $\mathbb{R}^Q$ whose support lies in M. Given $T \in \mathcal{R}_k(M; \mathbb{Z}_2),$ we denote by $\mathcal{F}(T)$ and $\textbf{M}(T)$ the \textit{flat norm} and the \textit{mass} of $T,$ respectively. Also, the \textit{support} of $T$ is denoted by $\mbox{spt}(T).$

Consider the following sets:
$$\mathcal{I}_k(M; \mathbb{Z}_2)=\{T \in \mathcal{R}_k(M; \mathbb{Z}_2): \partial T \in \mathcal{R}_{k-1}(M; \mathbb{Z}_2)\},$$
$$\mathcal{Z}_k(M; \mathbb{Z}_2)=\{T \in \mathcal{I}_k(M; \mathbb{Z}_2): \partial T=0\},$$
$$\mathcal{Z}_k(M, \partial M; \mathbb{Z}_2)=\{T \in \mathcal{I}_k(M; \mathbb{Z}_2):  \mbox{spt}(\partial T) \subset \partial M\},$$
and
$$\mathcal{Z}_{k, rel}(M, \partial M; \mathbb{Z}_2)= \mathcal{Z}_k(M, \partial M; \mathbb{Z}_2)/\mathcal{I}_k(\partial M; \mathbb{Z}_2).$$

The set $\mathcal{Z}_k(M; \mathbb{Z}_2)$ is the space of \textit{mod 2 (integral) fla}t $k$\textit{-cycles} in $M$ and we call the quotient space $\mathcal{Z}_{k, rel}(M, \partial M; \mathbb{Z}_2)$ as the space of \textit{relative (mod 2) flat cycles}. When $\partial M=\emptyset,$ we have that $\mathcal{Z}_{k, rel} (M, \partial M; \mathbb{Z}_2)$ is identical to $\mathcal{Z}_k (M; \mathbb{Z}_2).$

The support $\mbox{spt}([T])$ of a class $[T]\in \mathcal{Z}_{k, rel,}(M, \partial M; \mathbb{Z}_2)$ is defined by
$\mbox{spt}([T])= \bigcap_{T \in [T]} \mbox{spt}(T).$ Also,  the mass norm and flat norm in the space of relative cycles are defined, respectively, by
$$\textbf{M}([T])=\inf_{T \in [T]} \textbf{M}(T), \quad \mathcal{F}([T])= \inf_{T \in [T]} \mathcal{F}([T]),$$
for $[T]\in \mathcal{Z}_{k, rel,}(M, \partial M; \mathbb{Z}_2).$

We consider the space of relative flat cycles $\mathcal{Z}_{k, rel}(M, \partial M; \mathbb{Z}_2)$ endowed with the flat norm $\mathcal{F}$. When it is endowed with the topology of the mass norm, we denote it by $\mathcal{Z}_k (M, \partial M; \textbf{M}; \mathbb{Z}_2).$

Note that each $[T] \in \mathcal{Z}_{k, rel} (M, \partial M; \mathbb{Z}_2)$ has a unique \textit{canonical representative} $k$-chain $T^0 \in [T]$  such that $T^0\mres \partial M=0,$ in particular, $\textbf{M}([T])=\textbf{M}(T^0)$ and $\mbox{spt}([T])=\mbox{spt}(T^0),$ see \cite[Lemma 3.3]{Li-Zhou}. Also, it follows that $\mathcal{F}([T])\leq \textbf{M}([T])$. This canonical representative is obtained taking $T^0=S\mres (M \backslash \partial M)$ for any $S \in [T].$ To keep the notation simple we denote $[T]$ by $T.$

\subsection{Varifolds in manifolds with boundary}

The following definitions can be found in \cite{Pitts} and \cite{LeonSimon}. We denote by $R\mathcal{V}_k (M)$ the set of $k$\textit{-dimensional rectifiable varifolds} in ${\mathbb{R}}^Q$ with support contained in $M$ and equipped with the weak topology. Also $\mathcal{V}_k (M)$ is the closure of $R\mathcal{V}_k (M)$ in the weak topology.

 Given a varifold $V \in \mathcal{V}_k (M),$ the \textit{weight} and the \textit{support} of $V$ are denoted by $\|V\|$ and $\mbox{spt}\|V\|,$ respectively. Also, for $x \in \mbox{spt}\|V\|,$ we denote by $\mbox{VarTan}(V, x) \subset \mathcal{V}(\mathbb{R}^Q)$ as the set of  \textit{the varifold tangents of} $V$ \textit{at} $x,$ which is a natural generalization of tangent planes for smooth surfaces. 
 
 Given $V, W \in \mathcal{V}_k (M),$ the \textit{Pitts'} $\textbf{F}$-\textit{metric} is denoted by $\textbf{F}(V, W).$ This metric induces precisely the usual weak topology on the set $\{V \in \mathcal{V}_k(M): \|V\|(M)\leq L\},$ for each constant $L>0.$

If $R \subset M$ is a $k$-rectifiable set and $\theta$ is a $\mathcal{H}^k$-integrable non-negative  function on $R$, we denote by $\upsilon(R, \theta) \in \mathcal{V}_k (M)$ as being the \textit{rectifiable} $k$-\textit{varifold associated to} $R$ \textit{with multiplicity function} $\theta.$ If $\theta$ assumes only positive integers values, we say that $\upsilon(R, \theta)$ is an \textit{integral varifold}. We denote by  $I\mathcal{V}_k (M)$ the space of $k$\textit{-dimensional integral varifolds in} $M.$

Given $T \in \mathcal{R}_k (M; \mathbb{Z}_2),$ we denote by $|T| \in \mathcal{V}_k(M)$ the varifold induced by the support of $T$ and its coefficients. And for $T\in \mathcal{Z}_{k, rel,}(M, \partial M; \mathbb{Z}_2),$ we take $|T|=|T^0|.$

Given $V\in \mathcal{V}_k (M),$ let $X \in \mathfrak{X}_{tan} (M)$ be a generator of a one-parameter family of diffeomorphisms $\phi_t$ of $\mathbb{R}^Q$ with $\phi_0(M)=M,$ we have that the \textit{first variation of} $V$ \textit{along the vector field} $X$ is given by 
$$\delta V (X):=\left.\frac{d}{dt}\right|_{t=0} \textbf{M}((\phi_t)_{\sharp} V),$$
where $(\phi_t)_{\sharp} V$ is the \textit{pushfoward varifold of} $V$  (see \cite[39.2]{LeonSimon}).

\begin{definition} 
	Let $U \subset M$ be a relatively open subset. A varifold $V \in \mathcal{V}_k (M)$ is said to be \textit{stationary in} $U$ \textit{with free boundary} if $\delta V (X)=0$ for any $X \in \mathfrak{X}_{tan} (M)$ compactly supported in $U.$
\end{definition}

Note that a free boundary minimal submanifold is also stationary with free boundary. However, the reverse may not be true. 

 
 By the relative topology we consider the $k$\textit{-dimensional density}, $\Theta^k (V, x),$ of a varifold $V \in \mathcal{V}_k (M)$ as the density restricted to $M,$ that is, given $x \in M,$ we take the limit, if it exists,
$$\Theta^k(V, x):=\lim_{\rho \rightarrow 0} \frac{\|V\|(B_\rho (x) \cap M)}{\rho^k |B^k|},$$
 where $|B^k|$ is the volume of the $k$-dimensional unit Euclidean ball $B^k.$ For stationary varifolds the limit above always exists.
 
 For a fixed $x,$ define the function
 $$\Theta_x^k(V, \rho):=\frac{\|V\|(B_\rho (x) \cap M)}{\rho^k |B^k|}.$$
 In the case $\partial M=\emptyset,$ we have $B_{\rho}(x) \subset M$ and it is known that the function above for stationary varifolds satisfies the monotonicity formula \cite[Sections 17 and 40]{LeonSimon}: $\Theta_x^k(V, \rho)$ is non-decreasing in $\rho.$  Also, it is well known that any tangent varifold of a stationary varifold is a stationary Euclidean cone and $\Theta^k_x(C, \rho)=\Theta^k(V, x)$ for any $C \in \mbox{VarTan}(V, x)$ and for all $\rho >0.$  We write this fact as $\Theta^k_x(C, \infty)=\Theta^k(V, x).$

%

\subsection{Min-Max Definitions}

In the following we use the notions of homotopy as in \cite[Section 2]{Marques-Neves}, just replacing $\mathcal{Z}_{n}(M; \textbf{M}; \mathbb{Z}_2)$ by $\mathcal{Z}_{n, rel}(M,\partial M; \textbf{M}; \mathbb{Z}_2)$ in those definitions. Here we are taking $k= \mbox{dim} (M) -1=n,$ in the notations of the previous sections.

The set $[X, \mathcal{Z}_{n, rel}(M,\partial M; \textbf{M}; \mathbb{Z}_2)]^{\sharp}$ denotes the set of all equivalence classes of $(X, \textbf{M})$-homotopy classes of mappings into $\mathcal{Z}_{n, rel}(M,\partial M; \textbf{M}; \mathbb{Z}_2).$

Given an equivalence class $\Pi \in [X, \mathcal{Z}_{n, rel}(M,\partial M; \textbf{M}; \mathbb{Z}_2)]^{\sharp},$  each $S \in \Pi$ is given by $S=\{\phi_i\}_{i \in \mathbb{N}} $ for some $(X, \textbf{M})$-homotopy sequence of mappings $\{\phi_i\}_{i \in \mathbb{N}}$ into $\mathcal{Z}_{n, rel}(M,\partial M; \textbf{M}; \mathbb{Z}_2).$ We define
$$\textbf{L}(S)=\limsup_{i\rightarrow \infty} \max \{\textbf{M}(\phi_i(x)); x \in \mbox{dmn}(\phi_i)\}.$$

\begin{definition} 
	The \textit{width} of $\Pi$ is defined by
	$$\textbf{L}(\Pi)=\inf \{\textbf{L}(S): S \in \Pi\}.$$
\end{definition}

We say that $S \in \Pi$ is a \textit{critical sequence} for $\Pi$ if $\textbf{L}(S)=\textbf{L}(\Pi),$ and the \textit{critical set} $\textbf{C}(S)$ of a critical sequence $S$ is given by
$$\textbf{C}(S)=\textbf{K}(S) \cap \{V \in \mathcal{V}_n (M): \|V\|(M)=\textbf{L}(S)\},$$
where
$$\textbf{K}(S)=\Big\{V \in \mathcal{V}_n (M): V=\lim\limits_{j\rightarrow\infty}|\phi_{i_j}(x_j)| \ \mbox{as varifolds, for some subsequence}   $$
$$\{\phi_{i_j}\}\subset S \ \mbox{and} \ x_j \in \ \mbox{dmn}(\phi_{i_j}) \Big\}.$$

From \cite[Lemma 15.1]{Marques-Neves-Willmore} (see also \cite[4.1 (4)]{Pitts}) we know that there exist critical sequences for each class $\Pi$, and from \cite[4.2 (2)]{Pitts}, $\textbf{C}(S)$ is compact and non-empty.

\begin{definition}  \cite[Section 2.5]{Marques-Neves-Liokumovich} 
	 Let $X \subset I^m$ be a cubical subcomplex. We say that a continuous map in the flat topology $\Phi: X \rightarrow \mathcal{Z}_{n, rel}(M, \partial M; \mathbb{Z}_2)$ is a $p$-\textit{sweepout} if the $p$-th cup power of $\Phi^* (\overline{\lambda})$ is nonzero in $H^{p}(X; \mathbb{Z}_2)$, where $\overline{\lambda}$ is the generator of $H^{1}(\mathcal{Z}_{n, rel}(M, \partial M; \mathbb{Z}_2); \mathbb{Z}_2).$  
	\end{definition}
	

%
%

\begin{definition} 
A flat continuous map $\Phi: X \rightarrow \mathcal{Z}_{n, rel}(M, \partial M; \mathbb{Z}_2)$ has \textit{no concentration of mass} if
$$\lim_{r \rightarrow 0} \sup \{\|\Phi(x)\|(B_r (p)\backslash \partial M): x \in \mbox{dmn}(\Phi), p \in M\}=0.$$
The set of all $p$-sweepouts with no concentration of mass is denoted by $\mathcal{P}_{p}(M)$. 

\end{definition}

\begin{definition} 
	The $p$\textit{-width of} $M$ is given by
	$$\omega_{p}(M)=\inf_{\Phi \in \mathcal{P}_{p}(M)} \sup \{\textbf{M}(\Phi(x)): x \in \mbox{dmn}(\Phi)\}.$$
\end{definition}


\subsection{Min-Max Theorem}

\begin{definition}
	 Let $U \subset M$ be a relatively open subset, we say that a varifold $V \in \mathcal{V}_k$ is $\mathbb{Z}_2$-\textit{almost minimizing in} $U$ \textit{with free boundary} if for every $\epsilon>0$ we can find $\delta>0$ and $T \in \mathcal{Z}_{k, rel}(M,\partial M; \mathbb{Z}_2)$ with $\textbf{F} (V, |T|)<\epsilon$ and such that the following property holds true: if $T=T_0, T_1, \ldots, T_m\in \mathcal{Z}_{k, rel}(M,\partial M; \mathbb{Z}_2)$ with
	 \begin{itemize}
  \item[$\bullet$] $\mbox{spt}(T-T_i)\subset U$ for $i=1, \ldots, m;$
  \item[$\bullet$] $\mathcal{F}(T_i- T_{i-1})\leq \delta$ for $i=1, \ldots, m$ and
  \item[$\bullet$] $\textbf{M}(T_i)\leq \textbf{M}(T)+\delta$ for $i=1, \ldots, m$
\end{itemize}
then $\textbf{M}(T_m)\geq \textbf{M}(T)-\epsilon.$

\end{definition}

Roughly speaking, it means that we can approximate $V$ by a varifold induced from a current $T$ such that for any deformation of $T$ by a discrete family supported in $U,$ and with the mass not increasing too much (parameter $\delta$), then at the end of the deformation the mass cannot be deformed down too much (parameter $\epsilon$).

A varifold $V\in \mathcal{V}_k (M)$ is said to be $\mathbb{Z}_2$-\textit{almost minimizing in annuli with free boundary} if for each $p \in \mbox{spt}\|V\|$ there exists $r>0$ such that $V$ is $\mathbb{Z}_2$-almost minimizing in the annuli $M \cap A_{s, r} (p)=M \cap B_r (p) \backslash \overline{B}_s (p)$ for all $0<s<r$. If $p \notin \partial M,$ we require that $r < \mbox{dist} (p, \partial M).$ By Proposition \ref{Fermi.Convex.theo} (iii), this definition with respect to $A_{s, r}(p)$ or $\mathcal{A}_{s, r}(p)$ is equivalent. When $\partial M=\emptyset,$ we do not need use the expression `with free boundary'.

If $V \in \mathcal{V}_k (M)$ is $\mathbb{Z}_2$-almost minimizing in a relatively open set $U \subset M$ with free boundary, then $V$ is stationary in $U$ with free boundary (\cite{Pitts}, Th. 3.3).

The next result is a \textit{tightening process} to a critical sequence $S \in \Pi$ so that every $V\in \textbf{C}(S)$ becomes a stationary varifold with free boundary.

\begin{theorem} \label{sequencia.critica}
	Let $\Pi \in [X, \mathcal{Z}_{n, rel}(M,\partial M; \textbf{M}; \mathbb{Z}_2)]^{\sharp}$. For each critical sequence $S^* \in \Pi$, there exists another critical sequence $S \in \Pi$ such that $\textbf{C}(S)\subset \textbf{C}(S^*)$ and each $V \in \textbf{C}(S)$ is stationary in $M$ with free boundary.
\end{theorem}

\begin{proof}
The proof of this result is essentially the same as  \cite[Prop. 8.5]{Marques-Neves-Willmore}. The only modifications are the use of Th. 13.1 and 14.1 of \cite{Marques-Neves-Willmore}, as noted in \cite[Th. 4.17]{Li-Zhou}. In place of \cite[Th. 14.1]{Marques-Neves-Willmore} we use \cite[Th. 2.11]{Marques-Neves-Liokumovich}; and a compatible version of \cite[Th. 13.1]{Marques-Neves-Willmore} follows from \cite[Lemma A. 1]{Marques-Neves-Liokumovich} in the same way that the \cite[Th. 13.1]{Marques-Neves-Willmore} follows from  \cite[Lemma 13.4]{Marques-Neves-Willmore}.
\end{proof}

With the tightening process above we can prove the existence of a $\mathbb{Z}_2$-almost minimizing varifolds with free boundary such that it reaches the width of a chosen $(X; \textbf{M})$-homotopy class $\Pi \in [X, \mathcal{Z}_{n, rel}(M,\partial M; \textbf{M}; \mathbb{Z}_2)]^{\sharp}.$ When $\partial M= \emptyset,$ it was first proved by Pitts \cite[Th. 4.10]{Pitts} with maps in cubical domains for $1\leq k\leq n,$ and later by Marques and Neves \cite[Th. 2.9]{Marques-Neves} for cubical subcomplex domains when $k=n$. For the case with boundary, a version for cubical domains was proved by Li and Zhou \cite[Th. 4.21]{Li-Zhou}. We present below a version for the case $\partial M \neq \emptyset$ and take maps in cubical subcomplex domains when $k=n$.

\begin{theorem} \label{minmax.theorem}
	For any $\Pi \in [X, \mathcal{Z}_{n, rel}(M,\partial M; \textbf{M}; \mathbb{Z}_2)]^{\sharp}$, there exists $V \in I\mathcal{V}_n (M)$ such that
	\begin{itemize}
  \item[(i)] $\|V\|(M)=\textbf{L}(\Pi);$
  \item[(ii)] $V$ is stationary in $M$ with free boundary;
  \item[(iii)] $V$ is $\mathbb{Z}_2$-almost minimizing in small annuli with free boundary.
\end{itemize}
 
\end{theorem}

\begin{proof}
	Using the previous theorem, we can follow the same procedure in the proof of \cite[Th. 4.10]{Pitts} (see also \cite[Th. 4.21]{Li-Zhou}). To prove that $V$ is $\mathbb{Z}_2$-almost minimizing in small annuli with free boundary on $\partial M,$ just do as in the proof of \cite[Th. 4.21]{Li-Zhou}. 	
\end{proof}

We present now an important result that we use in the last section.

\begin{corollary} \label{approx.width}
	
For $p \in \mathbb{N}$ and each $\epsilon>0,$ we can find $V \in I\mathcal{V}_n (M)$ such that
	
		\begin{itemize}
  \item[(i)] $\omega_p(M)\leq\|V\|(M)\leq \omega_p(M)+\epsilon;$
  \item[(ii)] $V$ is stationary in $M$ with free boundary;
  \item[(iii)] $V$ is $\mathbb{Z}_2$-almost minimizing in small annuli with free boundary.
		\end{itemize}
		
\end{corollary}

\begin{proof}

Note that the results in Section 3.3 of \cite{Marques-Neves} can be extended for compact manifolds (with or without boundary) from the results in Section 2 of  \cite{Marques-Neves-Liokumovich}. So we can use the results from Section 3.3 of \cite{Marques-Neves}.

By definition we can find $\Phi: X \rightarrow \mathcal{Z}_{n, rel}(M, \partial M; \mathbb{Z}_2)$ a $p$-sweepout with no concentration of mass such that $\sup \{\textbf{M}(\Phi(x)): x\in \mbox{dmn}(\Phi)\}\leq \omega_p(M)+\epsilon.$ From Th. 3.6 of \cite{Marques-Neves} there exists an $(X, \textbf{M})$-homotopy sequence of mappings $S=\{\phi_i\}_{i \in \mathbb{N}} \in \Pi$ associated. By Th. 3.7 and Cor. 3.9 (ii) of \cite{Marques-Neves} we can extended this sequence to a sequence $\{\Phi_i\}_{i \in \mathbb{N}}$ of maps continuous in the mass norm and homotopics to $\Phi$ in the flat topology for large $i$. Moreover
$$L(\Pi) \leq L(S)=\lim_{i \rightarrow \infty} \sup \sup \{M(\Phi_i(x)): x \in X\}\leq \sup_{x \in X} M(\Phi(x)).$$

As $\Phi$ is a $p$-sweepout and $\Phi_i$ is flat continuous and homotopic to $\Phi$ for large $i,$ then $\Phi_i$ is also a $p$-sweepout for large $i$ with no concentration of mass by Lemma 3.5 of \cite{Marques-Neves}. Also from Cor. 3.9 (i) of \cite{Marques-Neves} we have that $\{\widetilde{\Phi}_i\}_{i \in \mathbb{N}} \in \mathcal{P}_p (M)$ for each $\widetilde{S}=\{\widetilde{\phi}_i\}_{i \in \mathbb{N}} \in \Pi$ and for large $i,$ where $\widetilde{\Phi}_i$ is the Almgren extension of $\widetilde{\phi}_i.$ Together with the above inequality we conclude that
$$\omega_p(M)\leq L(\Pi)\leq \sup_{x \in X} M(\Phi(x))\leq \omega_p(M)+\epsilon.$$

The remaining items are deduced from the above theorem. 	
\end{proof}

\section{One Dimensional Stationary Varifolds}
\label{sec:3}

In this section we prove some results related to stationary integral 1-varifolds. In particular, we prove some properties of free boundary geodesic networks. When $M$ is the unit disk  $B^2=\overline{B_1(0)}\subset \mathbb{R}^2,$ or a planar full ellipse $E^2 \subset \mathbb{R}^2$ sufficiently close to $B^2,$ we classify the free boundary finite geodesic networks, provided they are $\mathbb{Z}_2$-almost minimizing in annuli and have low mass. Also we prove our main theorem about regularity (Theorem \ref{theo.main.regularity}).

\subsection{Free Boundary Geodesic Networks}

Here we define certain stationary integral 1-varifolds whose support is given by geodesic segments. We follow the notations of  Aiex \cite{Aiex}.

\begin{definition}  \label{definition.geodesic.network}
	Let $U \subset M$ be a relatively open set. A varifold $V \in I\mathcal{V}_1(M)$ is called a \textit{(finite) geodesic network} in $U$ if there exist geodesic segments $\{\alpha_1, \ldots, \alpha_l\} \subset \mbox{int}(M)$ and $\{\theta_1, \ldots, \theta_l\}\subset \mathbb{Z}_{+}$ such that
	\begin{itemize}
    \item[(i)] $\displaystyle V\mres U=\sum_{i=1}^{l}v(\alpha_i \cap U, \theta_i);$
    \item[(ii)] The \textit{set of junctions} is the set $\displaystyle \Sigma_V=\bigcup_{i=1}^{l} (\partial \alpha_i) \cap U.$ Each $p \in \Sigma_V$ belongs to a set $\{\alpha_{i_1}, \ldots, \alpha_{i_m}\}$ for some $m=m(p) \in \mathbb{Z}_+,$ with $m\geq 3$ if $p \in \mbox{int}(M).$ If each of those geodesic segments is \-{parame}\-{te}\-{ri}\-{zed} by arc-length with initial point $p,$ then
    \begin{eqnarray} 
    	\sum_{k=1}^{m} \theta_{i_k} \dot{\alpha}_{i_k}(0)=0, \quad  \mbox{if}  \quad p \in \Sigma_V \cap \mbox{int}(M).\label{cond.geod.net.interior}
         \end{eqnarray}
The varifold $V$ above is called a \textit{(finite) free boundary geodesic network} in $U,$ if additionally holds
       \begin{eqnarray} 
       \sum_{k=1}^{m} \theta_{i_k} \dot{\alpha}_{i_k}(0)\perp  \partial M, \quad  \mbox{if}  \quad p \in \Sigma_V \cap \partial{M}. \label{cond.geod.net.boundary}
     \end{eqnarray}
        \end{itemize}

\end{definition}

A junction $p \in \Sigma_V \cap \mbox{int}(M)$ is said to be \textit{singular} in $\mbox{int}(M)$ if there exist at least two geodesic segments with $\theta_{i_k} \dot{\alpha}_{i_k}(0)\neq -\theta_{i_{k'}} \dot{\alpha}_{i_{k'}}(0),$ and \textit{regular} in $\mbox{int}(M)$ otherwise. In other words, an interior regular junction belong to the interior of each segment that contains it. When $p \in \Sigma_V \cap \partial M,$ we said that it is regular if $\dot{\alpha}_{i_k}(0)\perp \partial M$ for every $\alpha_{i_k}$ such that $p\in \partial\alpha_{i_k}.$ A \textit{triple junction} is a point $p \in \Sigma_V$ such that it belongs to exactly three geodesic segments with multiplicity one each. Obviously a triple junctions is not regular in $\mbox{int}(M).$

We can deduce the following properties as did in  \cite{Aiex}:

\begin{proposition} {\normalfont (Prop. 3.2 and Cor. 3.3 and 3.4 of \cite{Aiex}).} \label{geod.net.properties}
 Let $V$ be a free boundary geodesic network.
\begin{itemize}
  \item[(i)] $V$ is stationary in $U;$
  \item[(ii)] $\displaystyle \Theta^1 (V,x)=\sum_{k=1}^{m} \frac{\theta_{i_k}}{2} $ for $x\in \bigcap\limits_{k=1}^{m}v(\alpha_{i_k}\cap U, \theta_{i_k});$
  \item[(iii)] If $\Theta^1 (V, x)<2$ for all $x \in \mbox{spt} \|V\| \cap \mbox{int}(M),$ then every $p\in \Sigma_V \cap \mbox{int}(M)$ is a triple junction;
  \item[(iv)] If $\Theta^1(V, x)\leq 2$ for all $x \in \mbox{spt} \|V\| \cap \mbox{int}(M),$ then either $\Sigma_V \cap \mbox{int}(M)$ contains a triple junction, or each junction of $\Sigma_V \cap \mbox{int}(M)$  is regular and the geodesic segments that define a such junction have multiplicity one each;
  \item[(v)]  If $\Theta^1(V, x)\leq 1$ for $x \in \mbox{spt} \|V\| \cap \partial M,$ then a junction on $x$ is given by a geodesic segment with multiplicity one or two and orthogonal to $\partial M,$ or by two geodesic segments with multiplicity one each and with the same angles with respect to $\partial M.$
\end{itemize}

\end{proposition}

\subsection{Upper Bound for the Density}

Now we get an upper bound for the density in free boundary (finite) geodesic networks. This is similar  to Prop. 3.6 and Th. 3.7 from \cite{Aiex}, but with a different approach.

Following the notations of the Definition \ref{definition.geodesic.network}, Let $V \in V_1(M^2)$ and suppose that $V \mres U=\sum_{i,j} v(\alpha_{ij} \cap U, \theta_{ij}),$ not necessarily satisfying (\ref{cond.geod.net.interior}) or (\ref{cond.geod.net.boundary}), and here the density $\theta_{ij}$ of each geodesic segment $\alpha_{ij}$ is a positive real number. Supposing that the number of geodesic segments is finite, we call such varifold a \textit{generalized finite geodesic network}. Denote by $J_i$ the $i$-th junction of $V.$ For $M^2 \subset \left({\mathbb{R}}^2, \langle , \rangle \right),$ each segment $\alpha_{ij}$ of $V$ is determined by two junctions $J_i$ and $J_j$ such that $|a_{ij}|=|J_j-J_i.|$ In each $J_i$ we see that
$\dot{\alpha}_{ij}(0)=(J_j-J_i)/|J_j-J_i|.$ Also, in these notations, we have $\alpha_{ij}=\alpha_{ji}, \dot{\alpha}_{ij}(0)=-\dot{\alpha}_{ji}(0)$ and $|\dot{\alpha}_{ij}(0)|=1.$

Let $x \in {\mathbb{R}}^2$ such that
 \begin{eqnarray} \label{radial.stationary}
	\sum_i  \langle \theta_{ij}\dot{\alpha}_{ij}(0), J_j-x \rangle=0, \ \ \forall J_j \in \Sigma_V \cap \mbox{int}(M).
	\end{eqnarray}
	Obviously, if $V$ satisfies (\ref{cond.geod.net.interior}) then it satisfies  (\ref{radial.stationary}) for all $x \in {\mathbb{R}}^2.$

\begin{lemma} \label{upper.bound.Theorem}
	Let $M^{2}$ be a compact region in $\mathbb{R}^2$ with non-empty boundary and $V \in \mathcal{V}_1(M)$ be a generalized finite geodesic network such that it satisfies (\ref{radial.stationary}) for some $x \in {\mathbb{R}}^2.$ At each $J_l \in \Sigma_V \cap \partial M, J_l \neq x,$ such that $\sum_{i} \theta_{il} \dot{\alpha}_{il}(0) \neq 0,$ define 
	$$F_l:=\sum_{i} \theta_{il} \dot{\alpha}_{il}(0) \quad \mbox{and} \quad |F_l^x|:= |F_l| \cos(\phi_l^x), $$
	where $\phi_l^x=\angle(F_l, J_l-x).$ Then $\|V\|(M)=\sum_{l} |F_l^x| |J_l -x|.$ In particular, given $R>0,$ holds

\begin{enumerate}
   \item[(i)] $\displaystyle \|V\|(M)=R\sum_l |F_l^x|$ if $M^{2}=\overline{B_R(x)}.$
  \item[(ii)] $ \displaystyle R\sum_l |F_l^x| \rightarrow \|V\|(M)$ if the convergence $M^{2} \rightarrow \overline{B_R(x)}$ is smooth. Precisely: given $\varepsilon>0, \mathcal{C} >0,$ then for $M^2$ sufficiently $C^{\infty}$-close to $\overline{B_R(x)},$ we have that 
$$   \Big|\displaystyle  \|V\|(M)- R\sum\limits_{l}|F_l^x| \Big|<\varepsilon $$
for every free boundary geodesic network $V \in I\mathcal{V}_1(M)$ with $\|V\|(M)<\mathcal{C}.$

\end{enumerate}

\end{lemma}

\begin{proof}

Following the above notations, consider the index $l$ such that $J_l\in \partial M.$ We have that
\begin{eqnarray*}
	\|V\|(M)=\frac{1}{2}\sum_{i,j} \theta_{ij}|\alpha_{ij}|&=&\frac{1}{2}\sum_{i,j} \theta_{ij}|J_j-J_i|\\
	&=&\frac{1}{2}\sum_{i,j} \theta_{ij}(\langle \dot{\alpha}_{ij}(0), J_j-x \rangle+\langle \dot{\alpha}_{ji}(0), J_i -x\rangle)\\
	&=&\sum_{i,j}\langle  \theta_{ij} \dot{\alpha}_{ij}(0), J_j-x \rangle\\
	&=&\sum_{i,l} \langle \theta_{il}\dot{\alpha}_{il}(0), J_l-x \rangle\\
	&=&\sum_{l} \langle F_l,J_l -x \rangle\\
	&=&\sum_{l}  |F_l^x| |J_l -x|.
\end{eqnarray*}
Where we use (\ref{radial.stationary}) in the step to restrict the sum to junctions on the boundary.

If $M=\overline{B_R(x)},$ then $|J_l-x|=R$ for all $l$ such that $J_l \in \partial M.$ So, 
$$\|V\|(B_R(x))=R \sum_{l} |F_l^x|.$$

For $M$ close to $\overline{B_R(x)},$ we have $|J_l-x| \approx R$ for all $l$ such that $J_l \in \partial M.$ So,
$$\|V\|(M)= \sum_l |F_l^x| (R\pm \varepsilon_l). $$
Where $0<\varepsilon_l<\varepsilon_1$ and $\varepsilon_1 \rightarrow 0$ as $M \rightarrow \overline{B_R(x)}.$ Since $\|V\|(M)<\mathcal{C},$ we see by the above expression that $\sum_l |F_l^x|<\mathcal{C}_1$ for some constant $\mathcal{C}_1>0.$ Therefore, for $\varepsilon_1<\varepsilon/{\mathcal{C}_1},$ we obtain
$$ \Big| \|V\|(M)-R\sum_l |F_l^x| \Big| \leq \varepsilon_1 \sum_l |F_l^x| < \varepsilon_1 \mathcal{C}_1<\varepsilon. \vspace{-0.4cm}$$ 
\end{proof}


\begin{theorem}   \label{bound.density}
Let $V \in I\mathcal{V}_1 (B^2)$ be a free boundary geodesic network. Suppose that $\|V\|(B^2) \leq \mu$ for some positive real number $\mu,$ then 
\begin{enumerate}

\item[(i)] $\displaystyle \Theta^1(V,x) \leq \frac{\mu}{2}$ for all $x \in \mbox{int} (B^2) \cap \mbox{spt}\|V\|.$
\item[(ii)] $\displaystyle \Theta^1(V,x) < \frac{\mu}{2\sqrt{2}}$ for all $x \in \partial B^2 \cap \mbox{spt}\|V\|.$
\end{enumerate}

Futhermore, let $V \in I\mathcal{V}_1 (M^2)$ be a free boundary geodesic network and $0<\delta<\mu$ such that $\|V\|(M^2)\leq \mu-\delta,$ where $M^2$ is a compact region of $\mathbb{R}^2$ with strictly convex boundary. Then we can take $M$ sufficiently close to $B^2,$ whose approximation depends of $\mu$ and $\delta,$ such that the conclusions $(i)$ and $(ii)$ above are still true for $M^2$ in place of $B^2.$ 
\end{theorem}

\begin{proof}

Consider the case $M^2$ close to $B^2.$ We follow the notations above and we fix $x \in \mbox{spt}\|V\|.$ If $x \in \mbox{int}(M),$ we extend $V$ to a varifold $\widetilde{V}\in \mathcal{V}_1(\mathbb{R}^2)$ adding at each $J_l \in \Sigma_V \cap \partial M$ the semi-straight line $r^x_l$ starting at $J_l$ with direction $\dot{\alpha}^x_{l}(0):=(J_l-x)/|J_l-x|$ and multiplicity $\theta^x_l:=\langle F_l, -\dot{\alpha}^x_{l}(0) \rangle$ (see Fig. \ref{densit.formula} (a)). By the convexity of M, we note that $\langle F_l, -\dot{\alpha}^x_{l}(0) \rangle>0,$ $\forall l.$ If $x \in \partial M,$ then $x=J_k$ for some $J_k \in \Sigma_V \cap \partial M.$ In these case, we extend $V$ to a varifold $\widehat{V}\in \mathcal{V}_1(\mathbb{R}^2)$ adding at each $J_l \in \Sigma_V \cap \partial M, l \neq k,$ the semi-straight line $r^x_l$ with multiplicity $\theta^x_l$ as before, and at $J_k$ we add the  semi-straight lines $r^x_{ki}$ starting at $J_k $ with directions $-\dot{\alpha}_{ki}$ and multiplicities $\theta_{ki},$ respectively (see Fig. \ref{densit.formula} (b)).

\begin{figure}[ht]

\begin{center}

\includegraphics[trim=80 545 185 75,clip,scale=1]{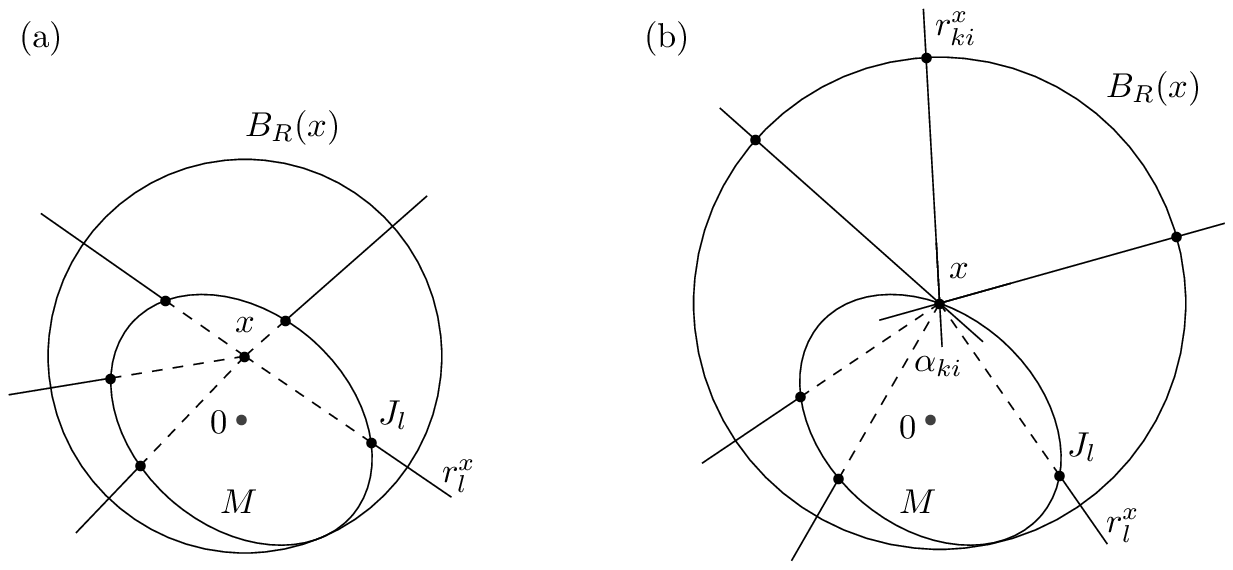}

\caption{}
	\label{densit.formula}
	
\end{center}
\end{figure}

By construction, $\widetilde{V}$ and $\widehat{V}$ satisfy (\ref{radial.stationary}) for this $x.$ In fact, $\widetilde{V}$ and $\widehat{V}$ are stationary in any relatively open set  $U \subset {\mathbb{R}}^2\backslash\{J_l \in \Sigma_V \cap \partial M : J_l\neq x \},$ therefore holds (\ref{radial.stationary}) for all $J_j \in \{ \Sigma_V \cap \mathbb{R}^2 \} \backslash\{J_l \in \Sigma_V \cap \partial M: J_l\neq x \}.$  And for each $J_l \in \Sigma_V \cap \partial M,$ $J_l\neq x,$ we have 
$$\Big \langle \sum_i  \theta_{il} \dot{\alpha}_{il}(0)+\theta^x_l \dot{\alpha}^x_{l}(0), J_l-x \Big \rangle =\langle F_l, J_l-x\rangle-\langle F_l, J_l-x\rangle=0. $$

Take the smallest $R_0>0$ such that $B_R(x)\supset M$ for all $R\geq R_0.$ For $R\geq R_0,$ let $r^x_l(R)=r^x_l \mres B_R (x)$ and $r^x_{ki}(R)=r^x_{ki} \mres B_R (x).$ Note that, $r^x_l(R)$ and $r^x_{ki}(R)$ do not intersect $M$ in $\mathbb{R}^2 \backslash \{J_l\}$ and in $\mathbb{R}^2 \backslash \{J_{k}\},$ respectively, since $M$ has strictly convex boundary.

\

\hspace{-0.47cm}(i) Suppose $x \in \mbox{spt}\|V\|\cap \mbox{int}(M).$ For $R\geq R_0,$ the monotonicity formula \footnote{ We can define $\Theta^1(\widetilde{V},x)$ and $\Theta^1(\widehat{V},x) ,$ since these varifolds have first null variation for radial directions with respect to the point $x.$ Also, $\Theta^1_x (\widetilde{V}, R)$ and $\Theta^1_x (\widehat{V}, R)$ are non-decreasing (see \cite[17]{LeonSimon}).} is given by
$$\Theta^1_x(\widetilde{V},R)=\frac{\| \widetilde{V}(B_R(x)) \|}{2R}=\frac{\|V\|(M)+\sum_l r^x_l(R)|F^x_l|}{2R}.$$
Note that $r^x_l(R)/R \rightarrow 1$ as $R \rightarrow \infty.$ Thus $\Theta^1_x(\widetilde{V},R) \rightarrow \sum_l |F^x_l|/2$ as $R \rightarrow \infty.$ For $\mathcal{C}=\mu-\delta,$ we can apply the above lemma taking $M$ close to $B^2$ such that $\big|\sum_l |F^x_l| - \|V\|(M) \big|<\delta.$ Therefore, for $R$ large
$$\Theta^1(\widetilde{V},x) \leq \Theta^1_x(\widetilde{V},R)<\frac{\|V\|(M)+\delta}{2}\leq\frac{\mu}{2}.$$
Where we used the fact that the function $\Theta^1_x(\widetilde{V},R)$ is non-decreasing for $x$ fixed. As $x \in \mbox{spt}\|V\|\cap \mbox{int}(M),$ we have that $\Theta^1(V,x)=\Theta^1(\widetilde{V},x)<\mu/2.$

For the case $M^2=B^2,$ we know that $\sum_l |F^x_l| = \|V\|(B^2)$ and taking $\mathcal{C}=\mu$ above, we get $\Theta^1(V,x)\leq\mu/2.$

\

\hspace{-0.47cm}(ii) Suppose now that $x \in \mbox{spt}\|V\|\cap \partial M.$ In this case the monotonicity formula is given by
$$\Theta^1_x(\widehat{V},R)=\frac{\| \widehat{V}(B_R(x)) \|}{2R}=\frac{\|V\|(M)+\sum_{l \neq k} r^x_l(R)|F^x_l|+\sum_i r^x_{ki}(R)|F^x_{ki}|}{2R}.$$
Where $|F^x_{ki}|=\theta_{ki}.$ Note that $ r^x_{ki}(R)=R.$ As above, we taking $R \rightarrow \infty$ and we get
$$\Theta^1(\widehat{V},x) \leq \frac{\sum_{l \neq k} |F^x_l|+\sum_i |F^x_{ki}|}{2}.$$
As $\sum_i |F^x_{ki}|=\Theta^1(\widehat{V},x)=2\Theta^1(V,x),$ we see that
 \begin{equation}  \label{resultantes.densidade}
 	\sum_{l\neq k } |F^x_l|  \geq 2\Theta^1(V,x).
 \end{equation}

By the above lemma $\|V\|(M)=\sum_{l\neq k } |F^x_l| |J_l-x|.$ In the case $M=B^2,$ note that $|J_l-x|=2 \cos(\phi^x_l),$ since $x \in \partial M.$ For $M$ close to $B^2$ we get that $|J_l-x|=2 \cos(\phi^x_l) \pm \delta_l,$ for all $x \in \partial M,$ where $0 \leq \delta_l \leq \delta_0$ and $\delta_0 \rightarrow 0$ as $M \rightarrow B^2.$ This follows from the compactness of $M$ and the fact that the convergence $M \rightarrow B^2$ is smooth. Therefore, for $M$ close to $B^2,$ 
\begin{eqnarray*}
	\|V\|(M)=\sum_{l\neq k } |F^x_l| \big(2 \cos(\phi^x_l) \pm \delta_l \big) &=& \sum_{l\neq k } \left( 2\frac{|F^x_l|^2}{|F_l|}\pm \delta_l |F^x_l| \right) \\
	& \geq & 2 \frac{\left( \sum_{l\neq k} |F^x_l| \right)^2}{\sum_{l\neq k }|F_l|}\pm \delta_l \sum_{l\neq k}|F^x_l|.
\end{eqnarray*}
Where the inequality follows from the Cauchy-Schwarz inequality. As $\|V\|$ is bounded, $\sum_{l\neq k} |F^x_l|$ is also bounded by the above lemma. Taking $|F^0_l|=\cos(\phi^0_l ) |F_l|,$ we note that $\cos(\phi^0_l)\approx 1$ for $M \approx B^2=\overline{B_1(0)}.$ By the above lemma, we can take $M$ close to $B^2=\overline{B_1(0)}$ such that $\big| \sum_{l} |F_l^0| - \|V\|(M) \big|<\delta.$ Therefore, using (\ref{resultantes.densidade}) in the last inequality,
$$2\big(2\Theta^1(V,x)\big)^2 < \|V\|(M)\big(\|V\|(M)+\delta\big) \pm \overline{\delta}.$$
Where $ \overline{\delta}\geq 0$ is such that $ \overline{\delta} \rightarrow 0$ as $M \rightarrow B^2.$ Since $\|V\|(M)\leq \mu-\delta,$ we can take $M$ close to $B^2,$ depending of $\mu$ and $\delta,$ such that
$$\Theta^1(V,x)< \frac{\mu}{2 \sqrt{2}}.$$

For the case $M^2=B^2,$ note that $\sum_l |F_l|=\sum_l |F^0_l|=\|V\|(B^2)$ and $\delta_0, \overline{\delta}=0$ in the above expressions.
\end{proof}

Note that the inequality (b) above is not sharp, since in the above proof we use that $\sum_{l\neq k} |F_l|< \sum_{l} |F_l|.$ The sharp inequality seems to be $\Theta^1(V,x)\leq \mu/4$ for $x \in \partial M \cap \mbox{spt}\|V\|.$ In fact, we can prove this for $\mu < 6$ assuming that $V$ is $\mathbb{Z}_2$-almost minimizing in small annuli with free \-{boun}\-{da}\-{ry}. We do not know any counterexample and we not discuss about this result in this article.

\subsection{Free Boundary Geodesic Networks with Low Mass}

In the following, we describe the free boundary geodesic networks with low mass and $\mathbb{Z}_2$-almost minimizing in annuli on the unit ball $B^2,$ and on full ellipses $E^2$ sufficiently close to $B^2.$ We need the following theorem:

\begin{theorem} \label{density.integer}
	{\normalfont(\cite{Aiex}, Th. 4.13)} Given $V \in I\mathcal{V}_1 (M)$ a geodesic network with free boundary and $p \in \Sigma_V \cap \mbox{int}(M)$. If $V$ is $\mathbb{Z}_2$-almost minimizing in annuli with free boundary at $p$, then
	$$\Theta^1(V,p)\in \mathbb{N}.$$
\end{theorem}

For $k\geq 3,$ let $P_k$ be a regular $k$-sided polygon inscribed in the unit circle. We consider $P_2$ as a diameter of the unit ball $B^2.$ Note that two regular $k$-sided polygons $P_k$ and $\widetilde{P}_k$ inscribed in the unit circle  are distinguished by a rotation. More generally, a $k$\textit{-polygon inscribed in a domain} $\Omega$ is a $k$-periodic billiard trajectory in  $\Omega,$ which is a periodic (billiard) trajectory obtained by $k$ reflexions at points of $\partial \Omega$ (the angle of incidence equals the angle of reflection).

\begin{theorem} \label{classif.geod.net.disc}
	Let $V \in I\mathcal{V}(B^2)$ be a free boundary geodesic network and $\mathbb{Z}_2$-almost \-{mi}\-{ni}\-{mi}\-{zing} in annuli with free boundary in $B^2.$ If $0<\|V\|(B^2)< 3 \sqrt{2},$ then $V=P_2$ or $V=P_2+\widetilde{P}_2.$
\end{theorem}


\begin{proof}
	From Theorem \ref{bound.density} we know that $\Theta^1(V, x)< 3 \sqrt{2}/2$ for $x \in \mbox{int}(B^2),$ and $\Theta^1(V, x) < 1.5$ for $x \in \partial B^2.$ Now using  Proposition \ref{geod.net.properties} (ii) and Theorem \ref{density.integer}, we deduce that $\Theta^1(V, x)=1$ or 2 for $x \in \mbox{int}(B^2)$, and $\Theta^1(V, x)=0.5,$ or $1$ for $x \in \partial B^2.$ Therefore, Proposition \ref{geod.net.properties} (iv) says that all junctions of $V$ in $\mbox{int}(B^2)$ are regular and the geodesic segments from each junction have  multiplicity one. Also, Proposition \ref{geod.net.properties} (iv) and (v) say that each segment of $V$ has multiplicity one or two and touches $\partial B^2$ orthogonally, or has multiplicity one and touches $\partial B^2$ making a reflexion and generating another segment with multiplicity one also. As $\|V\|(B^2)<3 \sqrt{2},$ we note that $V$ touches $\partial B^2$ orthogonally at some point, and we have that $V$ is a diameter $(V=P_2)$ or two diameters $(V=P_2+\widetilde{P}_2)$ of $B^2.$ Indeed, if $V$ does a reflexion at some point of $\partial B^2,$ then it contains a closed $k$-polygon inscribed in $B^2.$ A closed $k$-polygon in $B^2$ has all the sides with the same length and it is tangent to some circle $C_k$ concentric with $\partial B^2$ (see Fig. \ref{circle.2} (a), (b) and (c)), then the perimeter is at least $|C_k|.$ Each polygon $P_k$ gives a unique turn around $C_k.$  From five reflexions, we can have non-convex closed polygons as in the Fig. \ref{circle.2} (b). If $P_3 \subset V,$ then we have that $\|V\|(B^2)>3 \sqrt{2},$ and if $P_k \subset V,$ for $ k\geq 4,$ we observe that: if the radius of $C_k$ is bigger than 0.7, then the perimeter of a closed $k$-polygon is bigger than $2\cdot 0.7\pi>3 \sqrt{2}.$ Otherwise, if the radius of $C_k$ is less or equal to 0.7 (see Fig. \ref{circle.2} (d)), then each side of the closed $k$-polygon is bigger than 1.4, and so the perimeter is bigger than $4 \cdot 1.4>3 \sqrt{2}.$ Therefore, $V$ does not contain a closed $k$-polygon. 			
\end{proof}


\begin{figure}[ht]

\begin{center}

\includegraphics[trim=75 625 200 75,clip,scale=1]{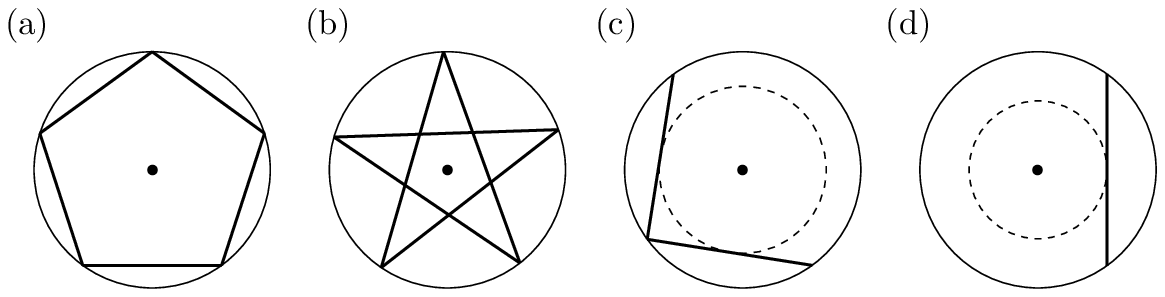}

\caption{}
	\label{circle.2}
	
\end{center}
\end{figure}

Let $E^2$ be a planar full ellipse $E^2.$ We denote by $P^E_k,$ for $k\geq 3,$ the closed convex $k$-polygon (not necessary regular) inscribed in $E^2.$ Here we consider $P^E_2$ as the smallest or the largest diameter of $E^2.$ As we see below, the smallest and the largest diameter of $E^2$ are the only $2$-polygons inscribed in $E^2.$

 The polygons $P^E_k$ are examples of closed periodic billiard trajectories in ellipses (Poncelet polygons). We see more properties of these polygons in the proof below. For instancy, given a point $A \in \partial E^2$ and $k\geq 3,$ there exists a unique $P^E_k$ that passes through $A$ (see the proof below). So, for a fixed integer $l \geq 3 ,$ is not difficult to see that if $E^2$ is close to $B^2,$ then a $k$-polygon $P^E_k$ is close to some $P_k,$ for $3\leq k \leq l.$ In fact, the boundary of $E^2$ is given by an ellipse $x^2/a^2+y^2/b^2=1$ and a point $A_0 \in \partial E^2$ is given in polar coordinates by $A_0=(a \cos(t_0), b\sin(t_0)),$ for some $t_0 \in [0,2 \pi).$ For $t_0$ fixed and $E^2 \rightarrow B^2,$ we take the $k$-polygons $P^E_k$  that pass through $A_0,$ defined by the points $A_0, A_1, \cdots, A_{k-1} \in \partial E^2,$ and such that $|P^E_k|= |\overline{A_0A_1}| + \cdots +|\overline{A_{k-2} A_{k-1}}|+|\overline{A_{k-1} A_{0}}|.$ Let $O=(0,0)\in \mathbb{R}^2,$ the segments $\overline{OA_i}$ tend to be perpendicular to $\partial E^2,$ as $a, b \rightarrow 1.$  Since at $A_i$ the angle of incidence is equals the angle of reflection, we see \footnote{For $i=0,$ we consider $A_{i-1}=A_{k-1},$ and for $i=k-1$ we consider $A_{i+1}=A_{0}.$} that $\angle(OA_i A_{i-1}) \approx \angle(OA_i A_{i+1}).$ Moreover,  $|\overline{A_jO}| \approx 1$ for all $j,$ so  $|\angle(A_iOA_{i-1}) - \angle(A_iOA_{i+1})|<\varepsilon,$ for small $\varepsilon > 0$ such that $\varepsilon \rightarrow 0$ as $E^2 \rightarrow B^2.$ In particular, $|2 \pi / k -\angle(A_j O A_{j+1})| =|\sum_{i=0}^{k-1} \angle(A_i O A_{i+1}) /k - \angle(A_j O A_{j+1})| < k \varepsilon$ for all $j.$ Finally, as $\partial E^2 \rightarrow \partial B^2,$ we conclude that these $k$-polygons $P^E_k$ tend to the regular $k$-polygon $P_k$ that passes through the point $(\cos(t_0), \sin(t_0)) \in \partial B^2.$

\begin{theorem} \label{classif.geod.net.ellip}
Let $E^2$ be a planar full ellipse and $0<R<3 \sqrt{2}$ be a real number. For $E^2$ sufficiently close to $B^2,$ depending only on the parameter $R,$ the following is true: if $V \in I\mathcal{V}_1 (E^2)$ is a free boundary geodesic network such that it is $\mathbb{Z}_2$-almost minimizing in annuli with free boundary in $E^2$ and $0<\|V\|(E^2)< R,$ then $V=P^E_2$ or $V=P^E_2+\widetilde{P}^E_2.$
\end{theorem}

\begin{proof}

Consider $E^2$ a planar full ellipse which boundary is given by an ellipse $x^2/a^2+y^2/b^2=1$ for $a > b$ with foci $F_1, F_2 \in Ox$ (see Fig. \ref{ellipses.1} (a)). Let $d$ and $D$ be the values of the smallest and largest diameters of $E^2,$ respectively. So, $d=2b$ and $D=2a.$ Also, here we are always considering $E^2$ sufficiently close to $B^2,$ so $d\approx D \approx 2,$ for example. 

Let $\mathcal{C}=R=3 \sqrt{2}-\delta,$ for some $\delta>0.$ We take $E^2 \approx B^2$ as in the \-{Theo}\-{rem} \ref{bound.density} and, as in the proof of the theorem above, applying Proposition \ref{geod.net.properties} (iv) and Theorem \ref{density.integer} to get: all junctions of $V$ in $\mbox{int}(E^2)$ are regular and the geodesic segments from each junction have multiplicity one; each segment of $V$ has multiplicity one or two and touches $\partial E^2$ orthogonally, or has multiplicity one and touches $\partial E^2$ making a reflexion and generating another segment with multiplicity one also. Therefore, $V$ can be the smallest or the largest diameters of $E^2,$ since they touch $\partial E^2$ \-{or}\-{tho}\-{go}\-{nally} (see Fig. \ref{ellipses.1} (a)). Also, $V$ can be $P^E_2+\widetilde{P}^E_2,$ and then $\|V\|(E^2)=2d, d+D$ or $2D,$ since $d\approx D \approx 2$ and $\|V\|<R<3 \sqrt{2}.$ We could have $V$ as in the Fig.  \ref{ellipses.1} (b): a segment \-{tou}\-{ching} $\partial E^2$ orthogonally at $A_1$, making a reflexion at $(0, b)\in \partial E^2$ with respect to $\partial E^2$ and generating another segment, which touches orthogonally $\partial E^2$ at $A_2=(-x(A_1), y(A_1)).$ This can  happen for $a>>b.$ However, for $E^2$ close to $B^2$ we have $a, b \approx 1,$ and the cases $V=P^E_2$ or $V=P^E_2+\widetilde{P}^E_2$ are the only possibilities such that $V$ touches $\partial M$ orthogonally in some point with $\|V\|(E^2)<R.$ Indeed, let $(a \cos(t), b \sin(t))$ be the polar coordinates on $\partial E^2$ for $t \in [0, 2 \pi),$ and take without loss of generality (by symmetry) $A\in \partial E^2$ such that $A=(a \cos (t_A), b \sin (t_A))$ for $t_A \in (3/4 \pi, 2 \pi).$ We claim that if a segment $\overline{AB} \subset E^2$ touches $\partial E^2$ orthogonally at $A,$ then $\overline{AB}$ is not orthogonal to $\partial E^2$ at $B \in \partial E^2,$ and the segment $\overline{BC},$ reflexion of $\overline{AB}$ at $B,$ is also not orthogonal to $\partial E^2$ at $C$  (see Fig. \ref{ellipses.1} (c)).


\begin{figure}[ht]

\begin{center}

\includegraphics[trim=70 630 225 75,clip,scale=1]{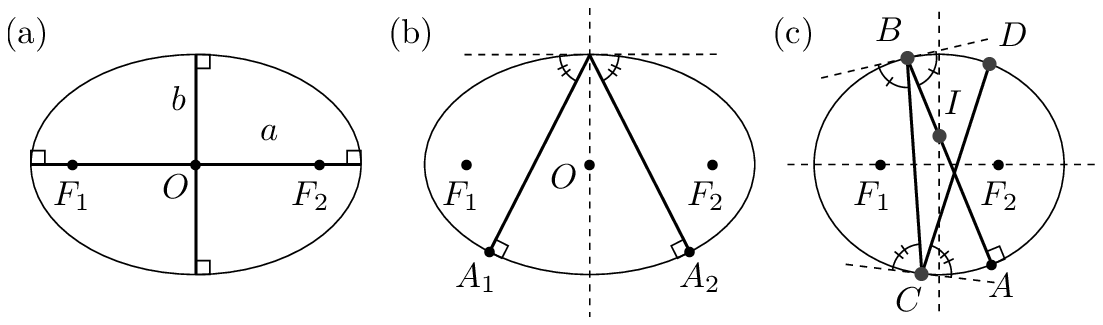}

\caption{}
	\label{ellipses.1}
	
\end{center}
\end{figure}

In fact, the equation of the straight line which is perpendicular to $\partial E^2$ at $A$ is given by
\begin{eqnarray*}
	y=\frac{a \tan(t_A)}{b}x+\sin(t_A)\left(b-\frac{a^2}{b}\right).
\end{eqnarray*}

If $\overline{AB}$ is orthogonal to $\partial E^2$ at $B=(a \cos(t_B), b \sin(t_B)),$ the equation of the straight line through $B$ is similar to above, this implies that $\tan(t_A)=\tan(t_B)$ and $\sin(t_A)=\sin(t_B),$ which is impossible, since $t_A \neq t_B.$ So, $\overline{AB}$ is not orthogonal to $\partial E^2$ at $B$ and there exists $\overline{BC},$ reflexion of $\overline{AB}$ at $B.$

Note that $a^2\leq 2 b^2,$ since $a, b \approx 1.$ So, for $x=0$ in the equation above, we see that $0<y(I)<b,$ where $I$ is the intersection of $\overline{AB}$ with $Oy$ (Fig.  \ref{ellipses.1} (c)). In an ellipse we have the following fact: if $\overline{AB}$ is orthogonal to $\partial E^2$ at $A,$ then $\overline{AB}$ bisects the angle $\angle F_1 AF_2.$ In particular, $\overline{AB}$ passes through $\overline{F_1 F_2}$ and, since $0<y(I)<b$, we have $t_B \in (\pi/2, \pi).$ Remember from billiard theory in ellipses that, if a segment in $E^2$ passes through $\overline{F_1F_2},$ then all the segments in that billiard trajectory (segments reflected at $\partial E^2$) pass through $\overline{F_1F_2}$ (see for example \cite[Th. 4]{Levi-Tabachnikov}). So $\overline{BC}$ passes through $\overline{F_1F_2}.$ 

Supposing that $\overline{BC}$ is orthogonal to $\partial E^2$ at $C,$ the same argument applied for $\overline{AB}$ can be applied to $\overline{BC}$ to get that $t_C \in (3/4 \pi, 2 \pi)$ and $t_C \neq t_A,$ where $C=(a \cos(t_C), b\sin(t_C)).$ Taking the equations of the straight lines that are perpendicular to $A$ and $C,$ respectively, we would have that they intersect at $B=(a \cos(t_B), b \sin(t_B)),$ then
$$ \frac{a^2}{b} \cos(t_B)(\tan(t_A)-\tan(t_C))+\left(\frac{b^2-a^2}{b}\right)(\sin(t_A)-\sin(t_C))=0. $$

As $t_A, t_C \in (3/4 \pi, 2 \pi),$ $t_A \neq t_C$ and $\cos(t_B), (b^2-a^2)<0,$ the left side of the last expression above is not equal to zero. Then, $\overline{BC}$ is not perpendicular to $\partial E^2$ at $C$ and there is another reflexion $\overline{CD}$ at $C$ (see Fig. \ref{ellipses.1} (c)).

Consider $E^2 \approx B^2$ such that each segment in $E^2$ through $\overline{F_1F_2}$ has length at least $R/3,$ since the length of each of these segments tending to 2 as $E^2$ tends to $B^2$ and $R<3 \sqrt{2}.$ By the above arguments, if $V\neq P^E_2$ and $V\neq P^E_2+\widetilde{P}^E_2,$ then $V$ has at least three segments, none of them is orthogonal to $\partial E^2$ and neither passes through $\overline{F_1F_2}.$ So, $V$ contains a closed $k$-polygon $\mathcal{P}_k,$ and moreover each segment is tangent to the same ellipse $\partial(E_k),$ where $E_k$ is a planar full ellipse inside of $\mathcal{P}_k$ and with the same foci of $E^2$ (see \cite[Th. 4]{Levi-Tabachnikov}). For simplicity, we just say that $\mathcal{P}_k$ is tangent to $\partial(E_k)$ 

The Poncelet theorem (see for instance \cite[Th. 4]{Ramirez-Ros}) says that if a closed $k$-polygon $\mathcal{P}_k$ is tangent to $\partial(E_k),$ then any other polygon $\mathcal{Q}$ that is  tangent to $\partial(E_k)$ is also a closed $k$-polygon with the same perimeter of $\mathcal{P}_k.$ Moreover, for each $k\geq 3$ there exists a unique $E_k$ such that all the convex closed $k$-polygons $P^E_k$ have its trajectory tangent to $\partial(E_k)$ (see for example \cite[Section 4]{Ramirez-Ros-Sonia}). In particular for a fixed $k\geq 3,$ all the polygons $P^E_k$ have the same perimeter.

Note that, given $A \in \partial E^2$ there is a unique $P^E_k$ through $A$ for each $k\geq 3.$ Indeed, just take the billiard trajectory starting at $A$ and tangent to $\partial E_k.$ Also, note that $|\partial(E_k)|<|\partial(E_{k+1})|$ since the tangency property of the polygons and the strict convexity of the ellipse $\partial (E)$  (see Fig. \ref{ellipses.3} (a)).


\begin{figure}[ht]

\begin{center}

\includegraphics[trim=75 633 240 80,clip,scale=1]{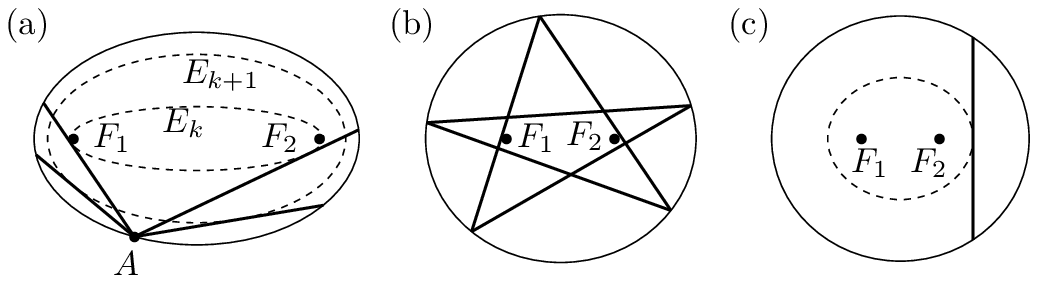}

\caption{}
	\label{ellipses.3} 
	
\end{center}
\end{figure}

We can take $E^2 \approx B^2$ such that $|P^E_3| \approx |P_3|,$ in particular $V$ does not contain $P^E_3,$ since $|P_3|> 3 \sqrt{2}.$ And if $V$  contains a closed $k$-polygon for $k \geq 4,$ we argue as in the proof of the theorem above. Indeed, the estimates in the accounts of the theorem above are strict, so for $E^2 \approx B^2$ and replace $C_k$ by $E_k$ with average radius approximately 0.7, we conclude that the perimeters are bigger than $3 \sqrt{2}$. Compare the Fig. \ref{circle.2} (d) and \ref{ellipses.3} (c). In the Fig. \ref{ellipses.3} (b) we have an example of a closed non-convex 5-polygon.		
\end{proof}

%
%
%

\subsection{Replacement and Regularity}

The regularity of stationary integral 1-varifolds for open sets was proven by Allard and Almgren (\cite{Allard-Almgren}, Section 3). As noted by Aiex (\cite{Aiex}, Th. 3.5), the regular structure described in \cite{Allard-Almgren} is exactly our definition of geodesic network. Precisely:

\begin{theorem} \label{regularidade.int} {\normalfont(\cite{Allard-Almgren}; \cite{Aiex}, Th. 3.5).} Let $M$ be a Riemannian manifold, $U \subset \mbox{int}(M)$ open and $K \subset U$ compact. If $V \in I\mathcal{V}_1 (M)$ is a stationary varifold in $U,$ then $V \mres K$ is a geodesic network. 
\end{theorem}

\begin{definition}  \label{mass.minim}
	Let $T \in \mathcal{Z}_{k} (M; \mathbb{Z}_2)$ and $U \subset M$ be a relatively open subset. We say that $T$ is \textit{locally mass minimizing in} $U$ if for every $p \in \mbox{spt}(T) \cap U$ there exists $r_p >0$ such that $B_{r_p}(p) \cap M \subset U$ and for all $S \in \mathcal{Z}_{k} (M; \mathbb{Z}_2)$ with $\mbox{spt}(T-S) \subset B_{r_p}(p) \cap M$ we have
	$$\textbf{M}(S)\geq \textbf{M}(T).$$
\end{definition}

By Proposition \ref{Fermi.Convex.theo} (iii), the definition above is equivalent if we take Fermi half-balls $\widetilde{\mathcal{B}}^+_{r_p}(p)$ instead of Euclidean balls $B_{r_p}(p)$ restricted to $M.$

The following theorem is about replacements of almost minimizing varifolds, which is one of the most important properties of this kind of varifolds. Roughly speaking, we can replace an almost minimizing varifold $V$ by another almost minimizing varifold $V^*,$ which has better regularity properties.

\begin{theorem} \label{replacement.theorem}
	Let $U \subset M$ be a relatively open set, $K \subset U$ compact and $V\in \mathcal{V}_k(M)$ be an $\mathbb{Z}_2$-almost minimizing varifold in $U$ with free boundary. There exists $V^* \in \mathcal{V}_k (M)$ such that
	\begin{itemize}
  \item[(i)] $V^* \mres (M\backslash K)=V\mres(M\backslash K);$
  \item[(ii)] $\|V^*\|(M)=\|V\|(M);$
  \item[(iii)] $V^*$ is $\mathbb{Z}_2$-almost minimizing in $U$ with free boundary;
   \item[(iv)] $V^* \in I\mathcal{V}(U \cap \mbox{int}(M));$
  \item[(v)] $V^* \mres U=\lim_{i \rightarrow \infty} |T_i|$ as varifolds for some $\{T_i\} \in \mathcal{Z}_{k, rel}(M, (M\backslash U)\cup \partial M; \mathbb{Z}_2)$ such that each $T^{0}_{i}$ is locally mass minimizing in $\mbox{int}_M (K).$
	\end{itemize}

\end{theorem}

\begin{proof}
	The proof follows as in Prop. 5.3 from \cite{Li-Zhou}, replacing Lemmas 3.10 and 3.7 by Th. 2.3 and Prop. 2.4 from \cite{Marques-Neves-Liokumovich}, respectively. See also Th. 3.11 and 3.13 from \cite{Pitts} to get (iv) from (iii). 	
\end{proof}

 The varifold $V^*$ in the above theorem is called of a \textit{replacement of} $V$ \textit{in} $K.$ See Section \ref{section2.1} to remember the notation of $T^{0}_{i}$ above.

 In the next lemma we prove a weak regularity of $V^* \in \mathcal{V}_1(M)$ for manifolds with strictly convex boundary.

\begin{lemma} {\normalfont (Weak Regularity of Replacements)} \label{regularidade.replacement}
Under the same hypotheses of Theorem \ref{replacement.theorem}, assume that $\partial M$ is strictly convex and take $V $ a one-dimensional varifold. Then $\mbox{spt}\|V^*\| \cap \mbox{int}_M (K)$ is a free boundary geodesic network (possibly infinite) without junctions in $(K \cap \mbox{int}(M))\backslash \partial_{rel} K,$ such that each geodesic segment has endpoints in $\partial_{rel} K \cup \partial M,$ and they can touch $\partial M \cap \mbox{int}_M (K)$  only orthogonally. 
	
\end{lemma}

\begin{proof}

From \cite[Prop. 4.6]{Aiex} we know that if $T$ is a one-cycle that is locally mass minimizing in an open set $W \subset \mbox{int}(M)$ and $Z \subset W$ is compact, then $T \mres Z$ is a geodesic network (finite) such that each geodesic segment has endpoints in $W \backslash Z$ and those segments do not intersect each other. So, for a relatively compact $K \subset M$ and $T^0_i$ locally mass minimizing in $\mbox{int}_M(K)$ (as in Theorem \ref{replacement.theorem}, (v)), we have that $T^0_i \mres \mbox{int}_M (K)$ is given by geodesic segments not intersecting each other, all segments have endpoints in $U\backslash K,$ and each segment that touches $\partial M \cap \mbox{int}_M (K)$ is orthogonal to $\partial M,$ in particular $|T^0_i| \mres \mbox{int}_M (K)$ is a free boundary geodesic network (possibly infinite). Indeed, as $T^0_i$ is locally mass minimizing, each segment of $T^0_i$ that touches $\partial M$ is locally the shortest path, so it is orthogonal to $\partial M.$

In the proof of Theorem \ref{replacement.theorem} we have that $\textbf{M}(T_i)$ is uniformly bounded, so we can use Th. 6.1 from \cite{Guang-Li-Zhou} to get that the limit in the Theorem \ref{replacement.theorem} (v) is smooth (after passing to a subsequence) for this lemma, then $V^*$ is given by geodesic segments such that each segment has endpoints in $\partial_{rel} K \cup \partial M,$ and they can touch $\partial M \cap \mbox{int}_M (K)$ only orthogonally. The last one follows from the fact that $\partial M$ is strictly convex, thus geodesic segments can touch $\partial M$ only in its endpoints. Also, as each $T_i$ only can touch $\partial M \cap \mbox{int}_M (K)$ orthogonally, therefore the same happens in the limit.
 
 Finally, as the segments of each $T_i$ do not intersect each other, we have that in the limit we do not have junctions. 
\end{proof}

We called the result above as weak regularity, because we do not know if the number of geodesic segments could be infinite. However, the above lemma is true for any codimension.

Let $p \in \mathbb{R}^2$ and let $C \in \mathcal{V}_1(\mathbb{R}^2)$ be a varifold such that $C=\sum_{i=1}^{l} v(r_i, m_i)$ for some $l, m_1, \cdots, m_l \in \mathbb{N},$ and each $r_i$ is some semi-straight line from $p.$ We say that $C$ is a \textit{cone with vertex at} $p.$ 

The next Lemma is very important to prove our main result about regularity. Essentially, we use it to glue replacements on overlapping annuli (see Step 2 in the proof of Theorem \ref{theo.main.regularity}).

\begin{lemma} \label{lemma.cone}
	Let $C \in I \mathcal{V}_1(\mathbb{R}^2)$ be a stationary cone with vertex at the origin $0 \in \mathbb{R}^2,$ and such that it is $\mathbb{Z}_2$-almost minimizing in $B_2(0) \subset \mathbb{R}^2.$ Then $C= v(r, m),$ for some $r$ a straight line passing through the origin $0,$ and for some  $m \in \mathbb{N}.$
\end{lemma}

\begin{proof}

We use the following fact: if $C$ is $\mathbb{Z}_2$-almost minimizing in $B_2(0),$ then each varifold tangent is also a stationary integral varifold on $T_x \mathbb{R}^2 \equiv \mathbb{R}^2$ such that it is $\mathbb{Z}_2$-almost minimizing in any bounded open subset of $\mathbb{R}^2$ \cite[Th. 3.11 and 3.12(1)]{Pitts}.

By Theorem \ref{density.integer} we have that $\Theta^1(C, 0)=k$ for some $k \in \mathbb{N}.$  

We prove the result by induction on $\Theta^1_0(C, \infty).$ Indeed, the result is obvious for $\Theta^1_0(C,\infty)\leq 1.$ Suppose that $\Theta^1_0(C, \infty) =k+1,$ and that the result is true for $\Theta^1_0(C, \infty)\leq k,$ $k\geq1.$ Let $C^*$ be a replacement of $C$ on $\overline{B}_1(0),$ we know that $C^*$ is integral, \-{sta}\-{tio}\-{na}\-{ry} and $\mathbb{Z}_2$-almost minimizing in $B_{2}(0).$ Also, $\|C^*\|(B_2(0))=\|C\|(B_2(0)),$  $C^*\mres (B_2(0) \backslash \overline{B}_1(0))=C\mres (B_2(0) \backslash \overline{B}_1(0)),$ and together with the monotonicity formula we get \vspace{-0.2cm}
$$\Theta^1_y(\mbox{VarTan}(C^*, y), \infty)=\Theta^1(C^*, y)\leq \lim_{\rho\rightarrow \infty}\Theta^1_y(C^*, \rho)=\Theta^1_0(C, \infty), \vspace{-0.2cm}$$
where $y \in \partial B_1(0) \cap \mbox{spt}\|C^*\|.$ 

We have two cases: $\Theta^1(C^*, y)= \lim_{\rho\rightarrow \infty} \Theta^1_y(C^*, \rho)$ for some $y \in \partial B_1(0) \cap \mbox{spt}\|C^*\|,$ or  $\lim_{\rho\rightarrow \infty}\Theta^1(C^*, y)< \Theta^1_y(C^*, \rho)$ for any $y \in \partial B_1(0) \cap \mbox{spt}\|C^*\|.$ In the first case, $C^*$ is a cone with vertex at $y.$  This implies that $C=v(r_y, m),$ for some $m \in \mathbb{N}$ and $r_y$ is the straight line that passes through $y$ and the origin, since  $C^*\mres (B_2(0) \backslash \overline{B}_1(0))=C\mres (B_2(0) \backslash \overline{B}_1(0)).$

In the second case, $\Theta^1_x(\mbox{VarTan}(C^*, y), \infty)\leq k$ for any $y \in \partial B_1(0) \cap \mbox{spt}\|C^*\|,$  since $ \Theta^1_0(C, \infty)=k+1.$ So, as $\mbox{VarTan}(C^*, y)$ is $\mathbb{Z}_2$-almost minimizing in $B_2(0),$ we can use the induction hypothesis for each $y$ to get that $\mbox{VarTan}(C^*, y)=v(r_y, m_y)$ for some $m_y \in \mathbb{N}$ and $r_y$ is the straight line that passes through $y$ and the origin. Using that $C^*\mres (B_2(0) \backslash \overline{B}_1(0))=C\mres (B_2(0) \backslash \overline{B}_1(0)),$ we conclude $C=v(r, m)$ for some $m \in \mathbb{N},$ and for some straight line $r$ through the origin.
\end{proof}

The next result is a boundary maximum principle for stationary varifolds with free boundary in codimension one case.

\begin{theorem} \label{Max.Princ.Boundary}
{\normalfont (Boundary maximum principle \cite[Th. 2.5]{Li-Zhou}).} 
Let $U \subset M^{n+1}$ be a relatively open subset and $V \in \mathcal{V}_n(M)$ be stationary with free boundary in $U.$ Suppose $N \subset \subset U$ is a relatively open connected subset in $M$ such that
\begin{itemize}
  \item[(i)] $\partial_{rel}N$ meets $\partial M$ orthogonally, if $\partial_{rel}N \cap \partial M \neq \emptyset;$
  \item[(ii)] $N$ is relatively strict convex in $M;$
  \item[(iii)] $\mbox{spt}\|V\| \subset \overline{N}.$
\end{itemize}
Then we have $\mbox{spt}\|V\| \cap \partial_{rel} N = \emptyset.$

\end{theorem}

Now we prove our main theorem about regularity of stationary $\mathbb{Z}_2$-almost minimizing varifolds with free boundary.

\begin{theorem} \label{theo.main.regularity}
Let $M^2$ be a compact Riemannian manifold with non-empty strictly convex boundary. If $V \in I\mathcal{V}_1 (M)$ is a stationary varifold with free boundary such that it is integral in $M$ and $\mathbb{Z}_2$-almost minimizing in small anulli with free boundary, then $V$ is a free boundary geodesic network.	
\end{theorem}

\begin{proof}

Here we follow similarly to the proof of  \cite[Th. 5.2]{Li-Zhou} and \cite[Prop. 6.3]{Colding-Lellis}, with the necessary modifications.

Given $p \in  \mbox{spt}\|V\| \cap \mbox{int}(M),$ we know by the Theorem \ref{regularidade.int} that in a small compact neighborhood around $p$ we have that $V$ is a geodesic network. So, assume that $p\in \mbox{spt}\|V\| \cap \partial M$ and fix $r>0$ such that 
\begin{eqnarray}\label{choose.r}
r<\frac{1}{4} \mbox{min} \{r_{Fermi}, r_{am}(p), r_{ort}(p)\},
\end{eqnarray}
where $r_{am}(p)>0 $ is such that $V$ is $\mathbb{Z}_2$-almost minimizing in $\mathcal{A}_{s,t}(p)$ with free boundary for all $0<s<t<r_{am},$ and $r_{ort}(p)>0$ is such that two distinct geodesics that are orthogonal to $\partial M \cap \widetilde{\mathcal{B}}^+_{\delta}(p)$ do not intersect each other in  $ \widetilde{\mathcal{B}}^+_{\delta}(p)$ for all $0<\delta<r_{ort}(p).$

Note that, as a consequence of the maximum principle (Theorem \ref{Max.Princ.Boundary}), we have the following: if $W \in \mathcal{V}_1 (M)$ is stationary in $\widetilde{\mathcal{B}}^+_{r}(p)$ with free boundary for $p \in \mbox{spt}\|W\| \cap \partial M$ and $r$ as above, then 
\begin{eqnarray}\label{max.princ.conseq.}
 \mbox{spt}\|W\| \cap \widetilde{\mathcal{S}}^+_t(p) \neq	\emptyset \quad \ \mbox{for all} \ 0<t\leq r.
\end{eqnarray}
In fact, there exists  $t_0 \in (0, r]$ the smallest number such that  $\mbox{spt}\|W \mres \widetilde{\mathcal{B}}^+_{r}(p) \| \subset \mbox{Clos}\big(\widetilde{\mathcal{B}}^+_{t_0}(p) \big).$ By the maximum principle we have that $\mbox{spt}\|W \mres \widetilde{\mathcal{B}}^+_{r}(p) \| \cap  \widetilde{\mathcal{S}}^+_{t_0}(p)=\emptyset,$ then $\mbox{spt}\|W \mres  \widetilde{\mathcal{B}}^+_{r}(p)\| \subset \mbox{Clos}\big(\widetilde{\mathcal{B}}^+_{t_1}(p)\big)$ for some $0<t_1< t_0,$ which is contradiction.  Using the same argument and suppose only that $W \neq 0$ in $\widetilde{\mathcal{B}}^+_r(p)$ for some $p \in \mbox{spt}\|W\| \cap \partial M,$ we conclude that there exists $0<\widetilde{t}<r$ such that
\begin{eqnarray}\label{max.princ.conseq.2}
 \mbox{spt}\|W\| \cap \widetilde{\mathcal{S}}^+_t(p) \neq	\emptyset \quad \mbox{for all} \ 0<\widetilde{t}<t\leq r.
\end{eqnarray}

\vspace{0.3cm}

\textbf{Step 1:} \textit{Constructing successive replacements on two overlapping annuli.}

\

Fix any $0<s<t<r.$ As $r<(1/4) r_{am}$ and $V$ is $\mathbb{Z}_2$-almost minimizing in $\mathcal{A}_{\widetilde{s},r_{am}/2}(p)$  with free boundary for all $0<\widetilde{s}<t<r_{am}/2,$ we can use the Theorem \ref{replacement.theorem} to get a first replacement $V^*$ of $V$ on $K=\overline{\mathcal{A}_{s,t}(p)}.$ The Lemma \ref{regularidade.replacement} says that
$$\Sigma_1:= \mbox{spt}\|V^*\| \cap \mathcal{A}_{s,t}(p)$$
is a free boundary geodesic network (possibly infinite). By Theorem \ref{replacement.theorem} (iii) we have that $V^*$ is still $\mathbb{Z}_2$-almost minimizing in $\mathcal{A}_{\widetilde{s},r_{am}/2}(p)$  with free boundary for all $0<\widetilde{s}<t<r_{am}/2,$ so we can apply again the Theorem \ref{replacement.theorem} to get a second replacement $V^{**}$ of $V^*$ on $K=\overline{\mathcal{A}_{s_1,s_2}(p)}$ for $0<s_1<s<s_2<t.$ Again, 
$$\Sigma_2:=\mbox{spt}\|V^{**}\| \cap \mathcal{A}_{s_1,s_2}(p)$$
is a free boundary geodesic network (possibly infinite). Let us consider the following choices: we fix any $s_1 \in (0, s),$ and we choose $s_2 \in (s, t)$ such that $\mbox{VarTan}(\Sigma_1, x)$ is a straight line transversal to $\widetilde{\mathcal{S}}^{+}_{s_2}(p)$ for all $x\in (\widetilde{\mathcal{S}}^{+}_{s_2}(p) \backslash \partial M),$ and $(\alpha \cap \widetilde{\mathcal{S}}^{+}_{s_2}(p)) \backslash \partial M \neq \emptyset$ for every geodesic segment $\alpha \in \Sigma_1.$ Indeed, fixing $s_2 \in (s, t),$ we know by the regularity of replacements  (Lemma \ref{regularidade.replacement})  that $\mbox{VarTan}(\Sigma_1, x)$ is a straight line for any $x \in \mathcal{A}_{s,t}(p).$ Also,  we  have only a finite number of geodesic segments $\{\alpha_i\}\subset \Sigma_1$ in $\mathcal{A}_{s, \widetilde{t}}(p)$ for any $0<s<\widetilde{t}<t.$ To see the last one, note that any geodesic segment (with possible multiplicity) $\alpha_i \in \Sigma_1 \cap \mathcal{A}_{s, t}(p)$ has to touch $\widetilde{\mathcal{S}}^+_{t}(p).$ Indeed, by the Lemma \ref{regularidade.replacement} each $\alpha_i$ has to touch $\widetilde{\mathcal{S}}^+_s(p) \cup \widetilde{\mathcal{S}}^+_{t}(p) \cup (\partial M \cap \mathcal{A}_{s, t}(p))$ and it can touch $\partial M \cap \mathcal{A}_{s, t}(p)$ only orthogonally. Using that any two orthogonal geodesic segments to $\partial M$ do not intersect each other in $\widetilde{\mathcal{B}}^+_r(p),$ together with the fact that $\widetilde{\mathcal{S}}^+_{s}(p)$ is strictly convex and orthogonal to $\partial M,$ we conclude that if $\alpha_i \in \Sigma_1 $ touches $\partial M \cap \mathcal{A}_{s, t}(p),$ then $\alpha_i \cap \widetilde{\mathcal{S}}^+_{s}(p) \neq \emptyset$ only if $\alpha_i$ touches $\widetilde{\mathcal{S}}^+_{s}(p) \cap \partial M$ (see Fig. \ref{sucess.replacem}). Also, if $\alpha_i$ does not touch $\widetilde{\mathcal{S}}^+_{t}(p),$ then its endpoints cannot be on $\partial M \cap \mathcal{A}_{s, t}(p),$ because $\alpha_i$ would be a stationary varifold with free boundary, contradicting  (\ref{max.princ.conseq.2}). Then, any $\alpha_i$ that touches $\widetilde{\mathcal{S}}^+_s(p)$ or $\partial M \cap \mathcal{A}_{s, t}(p),$ should touch $\widetilde{\mathcal{S}}^+_{t}(p).$  Therefore, if there is an infinite number of geodesic segments $\{\alpha_i\}\subset \Sigma_1$ in $\mathcal{A}_{s, \widetilde{t}}(p),$ then there are an infinite number of geodesic segments from $\widetilde{\mathcal{S}}^+_{\widetilde{t}}(p)$ to $\widetilde{\mathcal{S}}^+_{t}(p),$ contradicting the fact that $\Sigma_1$ has finite mass. Thus the set $\{\alpha_i\}$ is  finite. Finally, using again the strict convexity of $\widetilde{\mathcal{S}}^{+}_{s_2}(p),$ each geodesic segment that is tangent to $\widetilde{\mathcal{S}}^{+}_{s_2}(p)$ cannot touch $\widetilde{\mathcal{S}}^{+}_{\widetilde{s}_2}(p)$ for all $0<\widetilde{s}_2<s_2.$ So, by the finiteness of the geodesic segments and by (\ref{max.princ.conseq.}), we can choose $s_2 \in (s, t)$ as requested (see Fig. \ref{sucess.replacem}).


\begin{figure}[ht]

\begin{center}

\includegraphics[trim=85 585 195 70,clip,scale=1]{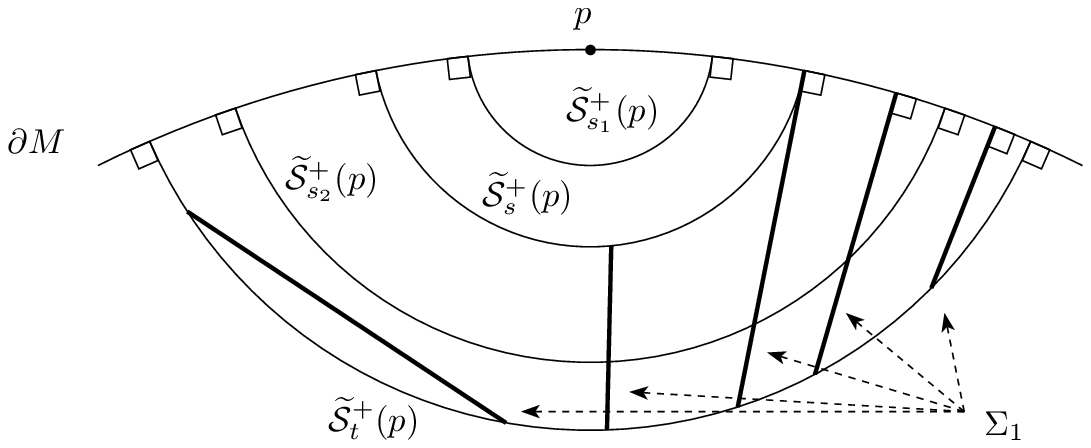}

\caption{}
	\label{sucess.replacem}
	
\end{center}
\end{figure}

Note that each $\alpha_i \subset \Sigma_1$ has to touch $\widetilde{\mathcal{S}}^+_t (p)$ at points in $\mbox{int}(M),$ since $\widetilde{\mathcal{S}}^+_t (p)$ is orthogonal to $\partial M.$

\vspace{0.5cm}

\textbf{Step 2:} \textit{Gluing} $\Sigma_1$ \textit{and} $\Sigma_2$  \textit{across} $\widetilde{\mathcal{S}}^+_{s_2}(p).$ 

\

As before, any geodesic segment (with possible multiplicity) $\beta_i \in \Sigma_2 \cap \mathcal{A}_{s, s_2}(p)$ has to touch $\widetilde{\mathcal{S}}^+_{s_2}(p)$ in points belonging to $\mbox{int}(M).$ Since  $V^{**}$ is stationary and integral in $\mathcal{A}_{s_1, t}(p),$ we have by the interior regularity (\-{Theo}\-{rem} \ref{regularidade.int}) that each $x \in \mbox{spt}\|V^{**}\| \cap \mbox{int}(M) \cap \mathcal{A}_{s, t}(p)$ belongs to a finite number of geodesic segments (including multiplicity). In particular, if $x \in \mbox{spt}\|V^{**}\| \cap \mbox{int}(M) \cap \widetilde{\mathcal{S}}^+_{s_2}(p)$ then $x$ belongs to $\overline{\Sigma}_1\cap \overline{\Sigma}_2,$ since each geodesic segment of $\Sigma_1$ touches $\widetilde{\mathcal{S}}^+_{s_2}(p)$ transversally. So, $\Sigma_1$ and $\Sigma_2$ glue continuously across $\widetilde{\mathcal{S}}^+_{s_2}(p).$ Note that $\mbox{spt}\|V^{**}\| \cap \widetilde{\mathcal{S}}^+_{s_2}(p)=\overline{\Sigma}_1\cap \widetilde{\mathcal{S}}^+_{s_2}(p)=\overline{\Sigma}_2\cap \widetilde{\mathcal{S}}^+_{s_2}(p) \subset \mbox{int}(M).$ Moreover, as $\mbox{VarTan}(V^{**}, x)$ is a cone satisfying Lemma \ref{lemma.cone}, we see that the gluing is actually $C^1$ (smooth) since $\mbox{VarTan}(V^{**}, x)$ is a straight line (with possible multiplicity).

\vspace{0.5cm}

\textbf{Step 3:} \textit{Unique continuation up to the point} $p.$

\

By Step 2 and property (i) of Theorem \ref{replacement.theorem}, we can extend $\Sigma_2$ to $\widetilde{\Sigma}_2$ in $\mathcal{A}_{s_1, t}(p)$ such that $\widetilde{\Sigma}_{2}=\Sigma_1$ on $\mathcal{A}_{s, t}(p),$ $\widetilde{\Sigma}_2$ is given by geodesic segments possibly with multiplicity and without interior junctions that can touch $\mathcal{A}_{s_1, t}(p)\cap \partial M$ only orthogonally, $\widetilde{\Sigma}_2 \mres \mathcal{A}_{s, s_2}(p)$ has a finite number of geodesic segments, and each geodesic segment of $\widetilde{\Sigma}_2$ has to touch $\widetilde{\mathcal{S}}^+_{t}(p).$  Using (\ref{max.princ.conseq.}), we can continue to take replacements in this way for all $0<s_1<s.$ For each $0<s_1<s$ as before, denote $\widetilde{\Sigma}_2$ by $\Sigma_{s_1}.$  If $0<s'_1<s_1<0,$ then we have that $\Sigma_{s'_1}=\Sigma_{s_1}$ on $\mathcal{A}_{s_1,t}(p).$ Thus
$$\Sigma:=\bigcup_{0<s_1<s} \Sigma_{s_1}$$
in $\widetilde{\mathcal{B}}^+_{t}(p)$ is given by geodesic segments   possibly with multiplicity and without interior junctions that can  touch $\partial M \cap (\widetilde{\mathcal{B}}^+_{t}(p)\backslash \{p\})$ orthogonally only, and each geodesic segment of $\Sigma$ has to touch $\widetilde{\mathcal{S}}^+_{t}(p).$ Moreover, $\Sigma\mres \widetilde{\mathcal{B}}^+_{\widetilde{t}}(p)$ has a finite number of segments for all $0<\widetilde{t}<t$ (see Fig. \ref{sigma.regular}).


\begin{figure}[ht]

\begin{center}

\includegraphics[trim=80 640 310 70,clip,scale=1]{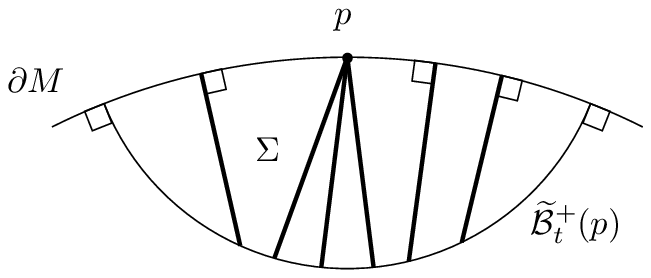}

\caption{}
		\label{sigma.regular}
	
\end{center}
\end{figure}

\textit{Claim:} $\mbox{spt}\|V\|=\Sigma$ \textit{in the punctured ball} $\widetilde{\mathcal{B}}^+_{s}(p)\backslash \{p\}.$

\

\textit{Proof of Claim:} Consider the set
\begin{eqnarray*}
T^V_p=\Big\{y \in \mbox{spt}\|V\|&:& \mbox{VarTan}(V, y) \ \mbox{is a straight line or} \\
&& \mbox{a semi-straight line transversal to} \ \widetilde{\mathcal{S}}^+_{\widetilde{r}_p (y)}(p) \Big\}.
\end{eqnarray*}
As in \cite[Claim 3, p. 42 ]{Li-Zhou}, we can use the convexity of small Fermi half-balls to apply a first variation argument as in \cite[Lemma B.2]{Colding-Lellis}, getting that the set $T^V_p$ is a dense subset of $\mbox{spt}\|V\| \cap  \widetilde{\mathcal{B}}^+_{s}(p).$ Note that the Lemma B from \cite{Colding-Lellis} holds, in the respective hypotheses, for $V \in I\mathcal{V}_n (M^{n+1}),$ $\forall n>0.$

Given $y \in T^V_p \cap (\widetilde{\mathcal{B}}^+_{s}(p)\backslash \{p\}),$ let $\rho=\widetilde{r}_p (y).$ Take $V^*$ the replacement of $V$ in $\mathcal{A}_{s,t}(p)$ and $V^{**}$ the replacement of $V^*$ in $\mathcal{A}_{\rho,s_2}(p)$ for $s_2 \in (s, t)$ chosen as in Step 1. By the property (i) from Th. \ref{replacement.theorem}, we have $V^{**}=V^*=V$ in $\widetilde{\mathcal{B}}^+_{\rho}(p),$ then
$$y \in \overline{\mbox{spt}\|V\|\cap \widetilde{\mathcal{B}}^+_{\rho}(p)}\cap \widetilde{\mathcal{S}}^+_{\rho}(p) =\overline{\mbox{spt}\|V^{**}\|\cap \widetilde{\mathcal{B}}^+_{\rho}(p)}\cap \widetilde{\mathcal{S}}^+_{\rho}(p). $$
Since $\mbox{spt}\|V^{**}\|=\Sigma$ in $\mathcal{A}_{\rho, t}(p)$ and $\mbox{VarTan}(V^{**},y)$ is transversal to $\widetilde{\mathcal{S}}^+_{\rho}(p),$ we have by (\ref{max.princ.conseq.}) and above that $y \in \Sigma.$ Thus, $T^V_p \cap (\widetilde{\mathcal{B}}^+_s(p)\backslash \{p\}) \subset \Sigma,$ and hence $\mbox{spt}\|V\| \cap (\widetilde{\mathcal{B}}^+_s(p)\backslash \{p\}) \subset \Sigma.$ The last one is deduced using that $T^V_p$ is a dense subset of $\mbox{spt}\|V\| \cap  \widetilde{\mathcal{B}}^+_{s}(p),$ and the fact that $\Sigma$ is compact in $\widetilde{\mathcal{B}}^+_{s}(p).$

To see the converse inclusion $\Sigma \subset \mbox{spt}\|V\|$ in $\widetilde{\mathcal{B}}^+_s (p),$ note that by the \-{Cons}\-{tan}\-{cy} Theorem \cite[Th. 41.1]{LeonSimon}, we have $\mbox{spt}\|V\|\cap (\widetilde{\mathcal{B}}^+_s (p) \backslash \{p\})=\Sigma$ in $M \backslash \partial M.$ For $y \in \Sigma \cap \partial M \cap (\widetilde{\mathcal{B}}^+_s (p)\backslash \{p\}),$ we know that $\mbox{VarTan}(\Sigma, y)$ is a straight line perpendicular to $T_y (\partial M),$ which implies that $y$ is a limit point of $\Sigma\cap \mbox{int}(M)$ and thus $y \in \mbox{spt}\|V\|.$ Therefore, $\mbox{spt}\|V\|\cap (\widetilde{\mathcal{B}}^+_s (p) \backslash \{p\})=\Sigma.$

\

\textbf{Step 4:} $V$ \textit{is a free boundary geodesic network}

\

From the interior regularity (Theorem \ref{regularidade.int}) and the Step 3, $V$ is a geodesic network (finite) in $  \widetilde{\mathcal{B}}^+_s (p)$ and a free boundary geodesic network (finite) in $(\widetilde{\mathcal{B}}^+_s (p)\backslash \{p\}).$ In \-{par}\-{ti}\-{cu}\-{lar,} $\Theta^1(V \mres \partial M, p)=0.$ So, if there exist geodesic segments at $p,$ as in the Fig. \ref{sigma.regular}, then those segments must  satisfy (\ref{cond.geod.net.boundary}), and then $V$ is a free boundary geodesic network (finite) in $\widetilde{\mathcal{B}}^+_s (p).$

Varying $p \in \mbox{spt}\|V\| \cap \partial M,$ we see that $V$ is a free boundary geodesic network (not necessarily finite) on $M.$ Given any compact $K \subset \mbox{int}(M),$ the interior regularity says that $V\mres K$ has a finite number of geodesic segments. So, we only need to find a compact $K\subset \mbox{int}(M)$ such that $V \mres (M\backslash K)$ has also a finite number of geodesic segments. Indeed, take a cover of $\mbox{spt}\|V\| \cap \partial M,$ by small open balls $\widetilde{\mathcal{B}}^+_s (p)$ as in the previous steps, extract a finite cover $\{\widetilde{\mathcal{B}}^+_{j} (p_j)\}_{j=1}^{l},$ and define $K:=\overline{M \backslash \big(\bigcup_{j=1}^{l} \widetilde{\mathcal{B}}^+_{j} (p_j)\big)}.$ Note that $K$ can be empty. This finishes the proof.
\end{proof}

\section{The Width of a Full Ellipse}
\label{sec:4}

In this section we prove our main theorem about $p$-widths: we calculate the first $p$-widths of $B^2$ and $E^2,$ where $E^2$ is a planar full ellipse $C^\infty$-close to $B^2.$ As in \cite{Aiex}, we take the $p$-sweepouts from Guth \cite[Section 6]{Guth}. We consider some adaptations to get a convenient upper bound for the mass of the cycles. Also, we need to take a better estimate than that given by the Cauchy-Crofton Formula. Indeed, to calculate the widths of the unit sphere in  \cite{Aiex}, the Cauchy-Crofton Formula gives a sharp estimate, which does not happen in our case.

\subsection{A Sweepout for $B^2$}

The sweepout that we use to calculate the $p$-widths is obtained by a map whose image is given by real algebraic varieties. The properties of this map can be found in Guth \cite[Section 6]{Guth}.

Let $Q_i: \mathbb{R}^2 \rightarrow \mathbb{R}$ denote the following polynomials for $i=1, \ldots, 4:$
$$Q_1(x,y)=x, \quad Q_2(x,y)=y, \quad Q_3(x,y)=x^2 \quad \mbox{and} \quad Q_4(x,y)=xy.$$

Also, put $A_p=\mbox{span}\left(\{1\} \bigcup_{i=1}^{p} Q_i\right)\backslash \{0\}$ and define the relation $Q \sim \lambda Q,$ for $\lambda \neq 0$ and $Q \in A_p.$ The quotient $(A_p, \sim)$ can be identified  with $\mathbb{RP}^p$ and by this identification we can define the map $F_p: \mathbb{RP}^p \rightarrow \mathcal{Z}_{1, rel} (B^2, \partial B^2; \mathbb{Z}_2),$ which send a class $[Q]$ to the real algebraic variety defined by $Q(x,y)=0$ restricted to $B^2,$ considered as a mod 2 relative Lipschitz cycle. As proved in \cite[Section 6]{Guth}, $F_p$ is a flat continuous map and it defines a $p$-sweepout. 

In the next lemma we use the Cauchy-Crofton formula to prove that $F_p$ has no concentration of mass, thus $F_p \in \mathcal{P}_p (B^2).$

\begin{lemma} \label{lemma.concentration.mass}
	The map $F_p: \mathbb{RP}^p \rightarrow \mathcal{Z}_{1, rel} (B^2, \partial B^2; \mathbb{Z}_2)$ has no concentration of mass for $p=1, \ldots, 4.$
\end{lemma}

\begin{proof}

Without loss of generality, consider $P_0=(p_0,0) \in B^2$ for $p_0\geq 0,$ and the ball $B_s (P_0)$ for $s>0$ sufficiently small.  Fixing $[Q] \in  \mathbb{RP}^p$ and recall that every straight line $r$ in the plane can be parameterized by the equation $x \cos(\theta)+y \sin(\theta)=\rho,$ where $\rho$ is the distance from $r$ to the origin and  $\theta \in [0, 2\pi)$ is the angle between the axis $Ox$ and the straight line that is perpendicular to $r$ and passes through the origin. Denote a such straight line by $r_{\rho, \theta}$ and let $n(\rho, \theta)$ be the number of intersection points (with multiplicity) of the straight line $r_{\rho, \theta}$ with $F_p([Q])$ in $B_s(P_0).$ 
	
If $p_0>0,$ note that for $\theta \in [0, \pi/2-\sin^{-1}(s/p_0)]$ the straight line $r_{\rho, \theta}$ intersects $B_s(P_0)$ if and only if $\rho \in  [p_0 \cos(\theta)-s, p_0 \cos(\theta)+s]$ (see Fig. \ref{limites.intersec.Fk} $(a)).$ On the other hand, if $\theta \in  (\pi/2+\sin^{-1}(s/p_0), \pi],$ then $r_{\rho, \theta}$ does not intersect $B_s(P_0) \cap B^2$ for all $\rho$ (see Fig. \ref{limites.intersec.Fk} $(b)).$ And for $\theta \in (\pi/2-\sin^{-1}(s/p_0), \pi/2+\sin^{-1}(s/p_0)),$ the straight line $r_{\rho, \theta}$ does not intersect $B_s(P_0)$ if $\rho > 2s.$


\begin{figure}[ht]

\begin{center}

\includegraphics[trim=75 580 230 80,clip,scale=1]{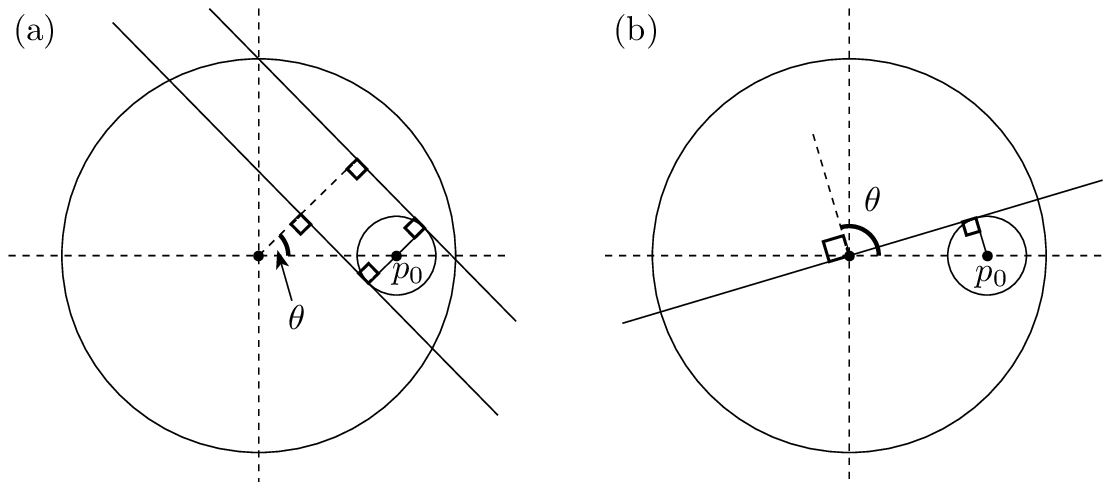}

\caption{}
			\label{limites.intersec.Fk}
	
\end{center}
\end{figure}

For $p=1, \ldots, 4,$ we have that $F_p([Q])$ is an algebraic variety of degree at most 2, so $F_p([Q])$ intersects $r_{\rho, \theta}$ at most two times. By the Cauchy-Crofton Formula we obtain
\begin{eqnarray*}
 \| F_p([Q])\| (B_s(P_0)\backslash \partial B^2) &\leq & \| F_p([Q])\| (B_s(P_0)) \\
 &=&\frac{1}{2}\int_{0}^{2 \pi}\int_{\mathbb{R^+}} n(\rho, \theta) d\rho d\theta \\
 &\leq &\frac{2}{2} \int_{0}^{\pi/2-\sin^{-1}(s/p_0)} \int_{p_0\cos(\theta)-s}^{p_0\cos(\theta)+s} 2 d\rho d\theta \\
  & &+ \frac{2}{2} \int_{\pi/2-\sin^{-1}(s/p_0)}^{\pi/2+\sin^{-1}(s/p_0)} \int_{0}^{2s} 2 d\rho d\theta \\
&=&4s \left(\frac{\pi}{2}+\sin^{-1}\left(\frac{s}{p_0}\right)\right).
\end{eqnarray*}

Similarly we have $\| F_p([Q])\| (B_s(P_0)\backslash \partial B^2) \leq 4 s \pi,$ when $p_0=0.$ Then, in all the cases we conclude that $\| F_p ([Q])\| (B_s(P_0)\backslash \partial B^2) \rightarrow 0$ as $s \rightarrow 0.$ 
\end{proof}

In the following, we estimate an upper bound for $\|F_p([Q])\|,$  $p=1, \ldots, 4.$ In other words, we estimate the maximum length of the algebraic variety $F_p([Q]).$ By the definitions above, $F_p([Q])$ is degenerate or is the restriction to $B^2$ of a straight line, or of two straight lines, or of a parabola, or  of a hyperbola. In other words, $F_p([Q])$ is a quadratic curve which is not an ellipse, since we excluded the monomial $Q_5(x, y)=y^2.$

\begin{lemma} \label{estimate.length}
	For any $[Q]\in \mathbb{RP}^p$ we have that $\|F_p([Q])\|\leq 2,$ $p=1,2,$ and  $\|F_p([Q])\| < 4.00267,$ $p=3, 4.$
\end{lemma}

\begin{proof}

Clearly, for $p=1, 2$ the algebraic variety  $F_p([Q])$ is degenerate or the restriction to $B^2$ of a straight line, thus $\|F_p([Q])\|\leq 2$ for $p=1, 2$ and for all $[Q] \in \mathbb{RP}^p.$

For $p=1, \ldots, 4$ note that if $F_p([Q])$ is degenerate or the intersection to $B^2$ of a straight line, or two straight lines, then $\|F_p([Q])\|\leq 4.$ Also, this estimate holds when $F_p([Q])$ is the restriction to $B^2$ of a hyperbola $H$ such that each branch intersects $B^2.$ Indeed, if we take $B_{r}(0)$ for $r \rightarrow \infty,$ we note that this hyperbola intersects $\partial B_r(0)$ in exactly four distinct points for all $r>1.$ In particular, as the diagonally opposite arms tend to the respective asymptote of those arms, we see that  for $r$ large the intersection of the two asymptotes is inside of $B_r(0)$ and each asymptote intersects $\partial B_r(0)$ in two points $z, z'$ (see Fig. \ref{hyperbola.estimate}). 


\begin{figure}[ht]

\begin{center}

\includegraphics[trim=90 590 330 70,clip,scale=1]{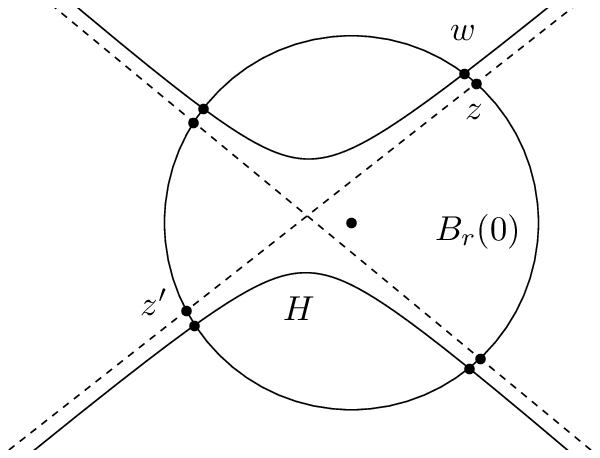}

\caption{}
	\label{hyperbola.estimate}
	
\end{center}
\end{figure}

Given a point $w \in H \cap \partial B_r(0)$ in a arm of a branch, take the respective half asymptote and consider the distance $|\overline{wz}|,$ where $z$ is the intersection between this half asymptote and $\partial B_r(0).$ Let $\varepsilon$ be the sum of the four distances given by the four points in $H \cap \partial B_r(0).$ Note that the length $L(H \cap B_r(0))$ of this hyperbola restricted to $B_{r}(0)$ is less than the length of the two asymptotes restricted to $B_{r}(0)$ added with $\varepsilon,$ so $L(H \cap B_r(0)) < 4r+\varepsilon.$  Now, decreasing $r$ to $r-s,$ $s \in (0,r-1],$ we note that the total reduction of length of the branches is at least $4s,$ since there exist four points in $H \cap \partial B^2$ during the reduction $r \rightarrow 1^{+}.$ Therefore $L(H \cap B^2) < 4+\varepsilon.$ Note that $\varepsilon \rightarrow 0$ as $r \rightarrow \infty,$ since each arm of the branches tends to their respective half asymptote. So, starting the reduction for $r$ as large as we want, we get that $L(H \cap B^2) < 4+\varepsilon$ for all $\varepsilon >0,$ that is, $L(H \cap B^2) \leq 4.$

In the other cases, hyperbolas with a unique branch intersecting $B^2$ or parabolas intersecting $B^2$, we prove in the Appendix that the maximum length of these curves restricted of the unity disk is bounded from above by approximately 4.00267, which concludes the lemma.
\end{proof}

\subsection{\textbf{The First Widths of $B^2$ and $E^2$}} \label{section.First.Widths}

Now, we prove our main result about $p$-widths: we calculate the low $p$-widths of the unit ball $B^2,$ and of full ellipses $C^{\infty}$-close to $B^2.$

\begin{theorem} \label{theo.main.widths}

 For $B^2$ we have 

\begin{itemize}
  \item[(i)] $\omega_1(B^2)=\omega_2(B^2)=2;$
  \item[(ii)] $\omega_3(B^2)=\omega_4(B^2)=4.$
\end{itemize}

Also, if $E^2$ is a full ellipse $C^{\infty}$-close to $B^2$ with small diameter $d$ and large diameter $D,$ then
\begin{itemize}
  \item[(iii)] $\omega_1(E^2)=d$ and $\omega_2(E^2)=D;$
  \item[(iv)] $\omega_3(E^2), \omega_4(E^2) \in \{2d, d+D, 2D\}.$
\end{itemize}

\end{theorem}

\begin{proof}
	
$(i)$ Let $p=1,2$ and take the $p$-sweepout $F_p\in \mathcal{P}_p(B^2).$ By Lemma \ref{estimate.length} we know that $\|F_p([Q])\|\leq 2$ for all $[Q] \in \mathbb{RP}^p,$ thus $\omega_1(B^2), \omega_2(B^2)\leq 2.$ Now, given $\epsilon>0$ we can find  by the Corollary \ref{approx.width} a special varifold $V$ such that $0<\omega_p(B^2)\leq \|V\|(B^2)\leq \omega_p(B^2)+\epsilon\leq 2+\epsilon.$ By Theorems \ref{theo.main.regularity}
 and  \ref{classif.geod.net.disc} we actually have that $V$ is a diameter of $B^2$ and $\|V\|(B^2)=2.$ Therefore, $\omega_1(B^2)=\omega_2(B^2)=2.$

$(ii)$ We observe that $\omega_3(B^2)>2$ as a consequence of the Lusternik-\-{Schni}\-{rel}\-{mann} theory (see for instancy Guth \cite{Guth}, p. 1923-24). Indeed, we can take three disjoint closed balls $B_i$ in $B^2 \backslash \partial B^2$ with radius $0.4$ each ball. Each $3$-sweepout $\Phi$ of $B^2$ is also an $1$-sweepout of $B^2,$ in particular it is an $1$-sweepout of each $B_i.$ The Lusternik-Schnirelmann theory says that $\Phi$ contains a cycle such that its mass is at least the sum of the first width of each $B_i.$ By the item (i) above we know that the first width of a ball is equal to the diameter of that ball, so $\omega_3(B^2)\geq 3 \times 0.8>2.$ Using Theorem \ref{classif.geod.net.disc} we get that $\omega_3(B^2)\geq 4$ (two diameters). Now, Lemma \ref{estimate.length} says that $4\leq\omega_3(B^2), \omega_4(B^2)< 4.003,$ and so by Corollary \ref{approx.width}, Theorems \ref{theo.main.regularity} and  \ref{classif.geod.net.disc}, we actually have that $\omega_3(B^2), \omega_4(B^2)= 4.$

$(iii)$ For $p=1,2$ and $E^2$ close to $B^2,$ we deduce by continuity and Lemma \ref{estimate.length} that $\omega_p(E^2)\leq 2+\delta$ for some small $\delta>0.$ Therefore, by Corollary \ref{approx.width},  Theorems \ref{theo.main.regularity}
 and \ref{classif.geod.net.ellip}, we conclude that the only possible values for $\omega_1(E^2)$ and $\omega_2(E^2)$ are $d$ or $D.$  Suppose that $\omega_1(E^2)=D$ for some $E^2.$ As $a^2\leq 2 b^2$ (notation in the proof of Theorem \ref{classif.geod.net.ellip}), we can take two small ellipses defined by scaling $E^2$ by half, taking a $\pi/2$ rotation and translating the variable $x$ by $+b/2$ and $-b/2,$ respectively. These two ellipses are inside of $\mbox{int}(E^2)$ and, as $\omega_1(E^2)=D,$ we have that the 1-width of each small ellipse is equal to $D/2.$ Using the Lusternik-Schnirelmann theory as before and the fact these two ellipses are inside of $\mbox{int}(E^2)$, we conclude that $\omega_2(E^2)>D/2+D/2=D,$ which contradicts the fact that $\omega_1(E^2), \omega_2(E^2)\in \{d, D\}.$ So,  $\omega_1(E^2)=d.$ Applying the same argument for $\omega_1(E^2)=d,$ we get that  $\omega_2(E^2)>d.$ Therefore, $\omega_2(E^2)=D.$
 
$(iv)$ We use again the continuity, Corollary \ref{approx.width},  Theorems \ref{theo.main.regularity}
 and \ref{classif.geod.net.ellip} to conclude that the only possible values to $\omega_3(E^2)$ and $\omega_4(E^2)$ are $2d, d+D$ or $2D.$
\end{proof}

%
%
%

%
%

\textbf{Acknowledgements}: I would like to deeply thank to Professor F. Marques for suggesting me to work on this problem and your support while visiting him at Princeton University. I am very thankful to Professors F. Vit\'{o}rio (PhD adviser), R. Montezuma and T.  Rivi\`{e}re by comments and suggestions. Also I am grateful to Department of Mathematic of Princeton University by its hospitality and where part of this work was done.

\section{Appendix}
\label{appendix}

In this appendix we prove the following result:

\begin{theorem} \label{comprimento.parabola}
	Let $L_0$ the maximum length of a parabola restricted to $B^2$ and $L_1$ the maximum length of a hyperbola restricted to $B^2.$ Then $L_1<L_0 \approx 4.00267.$ Moreover, there exists a unique parabola $\mathcal{P}_0$ such that $L(\mathcal{P}_0 \cap B^2)=L_0.$
\end{theorem}

In Rack \cite{Rack} was proved that $L_0\approx 4.00267.$ Since  we do not have direct access to \cite{Rack}, we give a geometric proof, and in our case we include the estimate of $L_1.$

The length of a real algebraic curve $C$ restricted to $B^2$ can be bounded in terms of its degree using the Cauchy-Crofton Formula (see Lemma \ref{lemma.concentration.mass}). In fact, if that curve has degree $d,$ then it intersects a straight line at most $d$ times, so  by the Cauchy-Crofton Formula the length of $C$ restricted to $B^2$ is at most $d \cdot \mbox{area}(B^2)=\pi d.$ Obviously, this upper bound is not sharp. For example, if $d=1$ we have that $C$ is a straight line and the length of the intersection of a straight line with $B^2$ is at most 2. It is intuitive, and it was conjectured in Guth \cite[p. 1974]{Guth}, that the general sharp upper bound is similar to the case $d=1,$ that is, $L(C \cap B^2)\leq 2d$ for all $d \in \mathbb{N}.$ Contradicting that result, for $d=2$ we can find $C$ such that $L(C\cap B^2)>4.$ Our counterexample is the parabola $\mathcal{P}_0$ from the above theorem.

\vspace{0.3cm}

 \hspace{-0.47cm}\textit{Proof of Theorem \ref{comprimento.parabola}.} By the proof of Lemma \ref{estimate.length}, we know that if $\gamma$ is a hyperbola such that each branch intersects $B^2,$ then $L(\gamma \cap B^2) \leq 4.$ Thus, from now consider $\gamma$ a branch of a hyperbola or a parabola. We choose an orientation such that the axis of symmetry of that curve is orthogonal to $x$-axis, and such that $\gamma$ is convex downward. So, $\gamma$ is a function of $x$ with a global minimum at the vertex $V,$ it is strictly increasing for $x>x(V),$ and strictly decreasing for $x<x(V).$ Moreover, the curvature  is strictly increasing in the direction of the axis of symmetry, and there exist at most four points in the intersection $\gamma \cap \partial B^2.$ We fix a such curve $\gamma$ such that $L(\gamma \cap B^2)>0,$ and by translation we find the positions such that the length $L(\gamma \cap B^2)$ increases, next we change the parameters of that curve to get the maximum of $L(\gamma \cap B^2).$

 As $\gamma$ is convex downward and $\partial B^2 \cap \{(x,y) \in \mathbb{R}^2 : y >0\}$ is convex upward, we conclude that there exist at most two points $A, D \in \gamma \cap \partial B^2$ such that $y(A), y(D)>0.$ So, consider two cases: there exist two points $A, D \in \gamma \cap \partial B^2$ such that $y(A), y(D)>0;$ or there exists at most one such point. In the first case, as in the examples of the Fig. \ref{first.estimate} (a), take $B=(x(A), -y(A)), C=(x(D), -y(D))\in \partial B^2,$  and the circular arc $\wideparen{BC}.$ The length $L(\gamma \cap B^2)$  is at most $|\overline{AB}|+|\wideparen{BC}|+|\overline{CD}|.$ Let $\alpha$ (resp. $\beta$) be the angle between $\overline{OA}$ (resp. $\overline{OD}$) and $x$-axis for $\alpha, \beta \in (0, \pi/2],$ then
\begin{eqnarray*} 
	L(\gamma \cap B^2)\leq |\overline{AB}|+|\overline{CD}|+|\wideparen{BC}|\leq 2 \sin(\alpha)+2 \sin(\beta)+\pi-(\alpha+\beta).  
\end{eqnarray*}


\begin{figure}[ht]

\begin{center}

\includegraphics[trim=75 570 205 80,clip,scale=1]{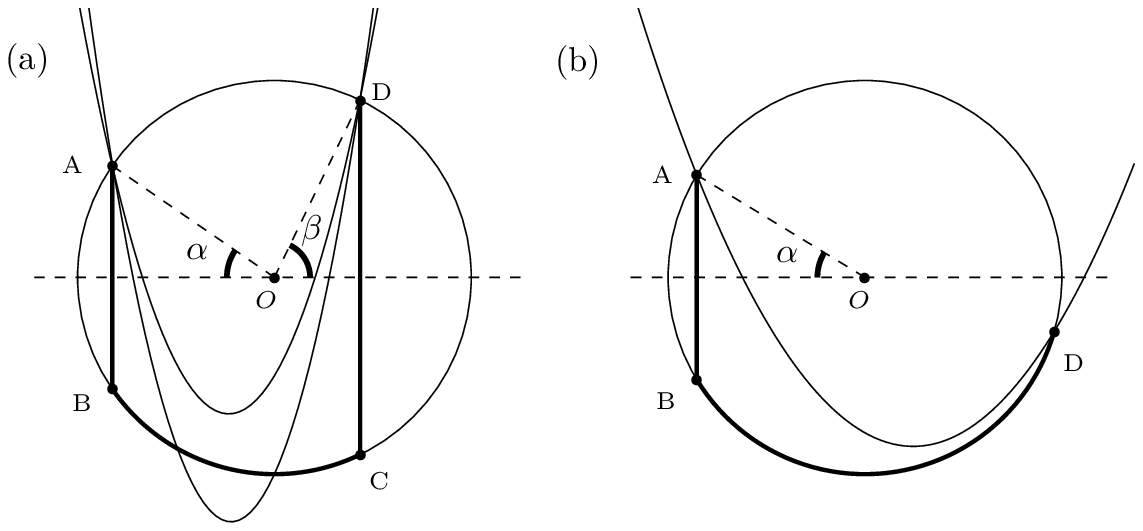}

\caption{}
    \label{first.estimate}
	
\end{center}
\end{figure}

In the second case, as in the example of the Fig. \ref{first.estimate} $(b),$ where do not exist $A$ or $D$ as in the first case, we take $\alpha=0$ or $\beta=0$ in the above estimate, respectively. Without loss of generality suppose $\beta=0.$ As $\alpha \in [0,\pi/2],$ we get
\begin{eqnarray} \label{primeira.estimativa1}
L(\gamma \cap B^2)\leq 2 \sin(\alpha)+\pi-(\alpha) \leq 2 \frac{\sqrt{3}}{2}+\frac{2 \pi}{3} <4.
 \end{eqnarray}

So, from now consider that there exist two points at the intersection between $\gamma \cap B^2$ and the upper half plane $\mathcal{H}^+:=\{(x,y) \in \mathbb{R}^2 : y >0\}.$

Suppose that $V \notin B^2$ and $\gamma \cap B^2$ is connected. Let $A \in \partial B^2$ be first point of contact between $\gamma$ and $B^2,$ and $D \in \partial B^2$ be last point of contact, $x(A)<x(D).$ Thus, $\gamma \subset B^2$ between the points $A$ and $D.$ In particular, $x(V)<x(A)$ or $x(V)>x(D).$ Because of symmetry, we can assume without loss of generality that $x(V)<x(A).$ So $\gamma \cap B^2$  is strictly increasing, in particular $y(A)<y(D).$ Note that $\gamma$ is contained within the triangle $\triangle ADE$ between the points $A$ and $D,$ where $E=(x(D), y(A)).$ Also, by the conditions on the points $A$ and $D,$ we see that the intersection $\triangle ADE \cap \gamma \cap B^2 \cap \mathcal{H}^+$ is empty, or is the point $D,$ or are the points $A$ and $D.$ Therefore, by the assumption of the previous paragraph, we necessarily have that $y(A)>0, y(D)>0$ (see Fig. \ref{vertice.fora.conexo} (a)). Let $\beta$ be the counterclockwise angle between the $x$-axis and $\overline{OD},$ and let $\alpha$ be the angle between $\overline{OA}$ and the $x$-axis. Note that $\alpha \in (0, \pi/2),$ $\beta \in (0,\pi).$ As $L(\gamma\cap B^2)$ is bounded by $|\overline{AE}|+|\overline{DE}|,$ we get
\begin{eqnarray*}
	 L(\gamma \cap B^2) \leq |\overline{AE}|+|\overline{DE}|&=&\cos(\alpha)+\cos(\beta)+\sin(\beta)-\sin(\alpha) \\
	 &=& (\cos(\alpha)-\sin(\alpha))+(\cos(\beta)+\sin(\beta))<3.
\end{eqnarray*}

\begin{figure}[ht]

\begin{center}

\includegraphics[trim=83 565 185 75,clip,scale=1]{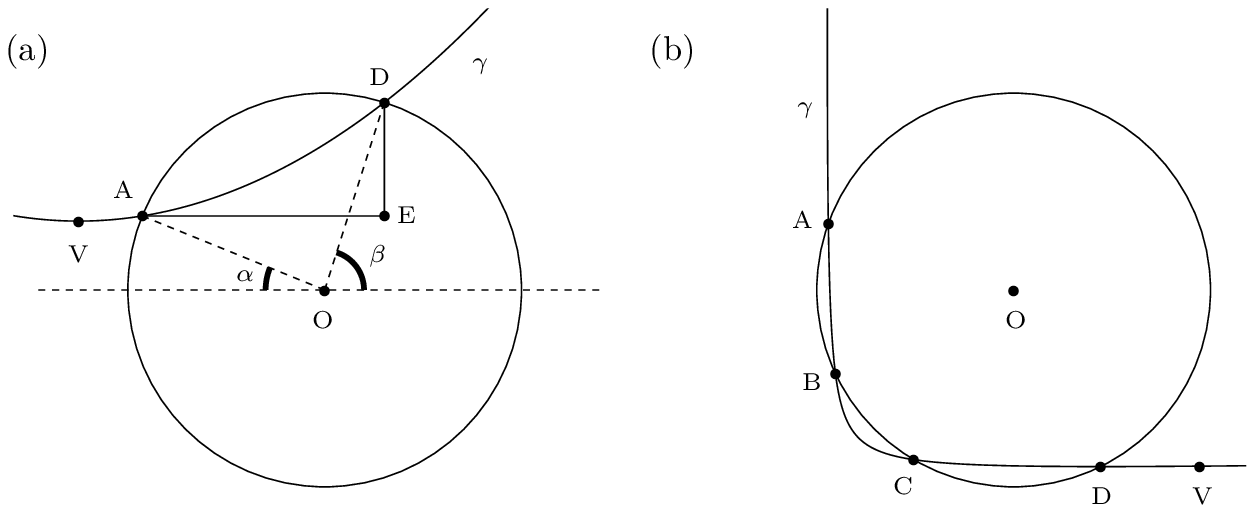}

\caption{}
   \label{vertice.fora.conexo}
	
\end{center}
\end{figure}

In the following arguments we see that to get the maximum length of $\gamma \cap B^2$ by translations, it is necessary that $V \in B^2.$
 
  Suppose now that $V \notin B^2$ and $\gamma \cap B^2$ is not connected. As the intersection $\gamma \cap \partial B^2$ has at most four points, $\gamma \cap B^2$ has at most two connected components $\gamma_1$  and $\gamma_2$ such that $L(\gamma_1)>0$ and $L(\gamma_2)>0.$ If there exists only one such connected component, which goes inside $B^2$ at $A \in \partial B^2$ and it goes outside $B^2$ at $D \in \partial B^2,$ then we can use the last estimate if $y(A), y(D)>0,$ or the estimate (\ref{primeira.estimativa1}) in all other cases to get that $L(\gamma \cap B^2)<4.$ If there exist two such connected components, then there exist two points $A, C \in \partial B^2$ where the curve $\gamma$ goes inside $B^2,$ and two points $B, D \in \partial B^2$ where the curve goes outside $B^2.$ Supposing that $x(A)<x(B)<x(C)<x(D),$ we claim that $x(B)<x(V)<x(C).$ Otherwise, as $V \notin B^2,$ we have that $x(V)<x(A),$ or $x(D)<x(V).$ By symmetry, it is enough to verify that the second inequality cannot be true (see Fig. \ref{vertice.fora.conexo} (b)). Indeed, as $\gamma$ goes outside $B^2$ at $B$ and goes inside $B^2$ at $C,$ there exists $x' \in (x(B),x(C))$ such that $\kappa_{\gamma}(x')>\kappa_{\partial B^2}=1.$ In \-{par}\-{ti}\-{cu}\-{lar,}  $\kappa_{\gamma}(x)>1$ for all $x \in [x', x(V)),$ since the curvature $\kappa_{\gamma}(x)$ of the $\gamma(x)$ is increasing for $x<x(V),$ this contradicts the fact that $\gamma$ goes inside $B^2$ at $C$ and going outside $B^2$ at $D,$ since  $x(D)<x(V)$ and  $\kappa_{\partial B^2}=1.$ Therefore, $x(B)<x(V)<x(C),$ also $y(A)>y(B)$ and $y(C)<y(D).$ As we are supposing that $\gamma \cap B^2 \cap \mathcal{H}^+$ has two points, it is not difficult to see that $y(A), y(D)>0$ and $y(B), y(C)<0.$ It follows as in the previous triangle argument. Now we have (by symmetry) two cases: $\overline{CD}$ is on the left of $O;$ or $\overline{AB}$ is on the left of $O,$ and $\overline{CD}$ is on the right of $O$ (see Figs. \ref{ontheleft} and \ref{no.meio}). Here, a segment is said on the left of $O$ (resp. on the right of $O$) if it intersects the $x$-axis for $x \leq 0$ (resp. $x>0$).


  \begin{figure}[ht]

\begin{center}

\includegraphics[trim=80 570 185 75,clip,scale=1]{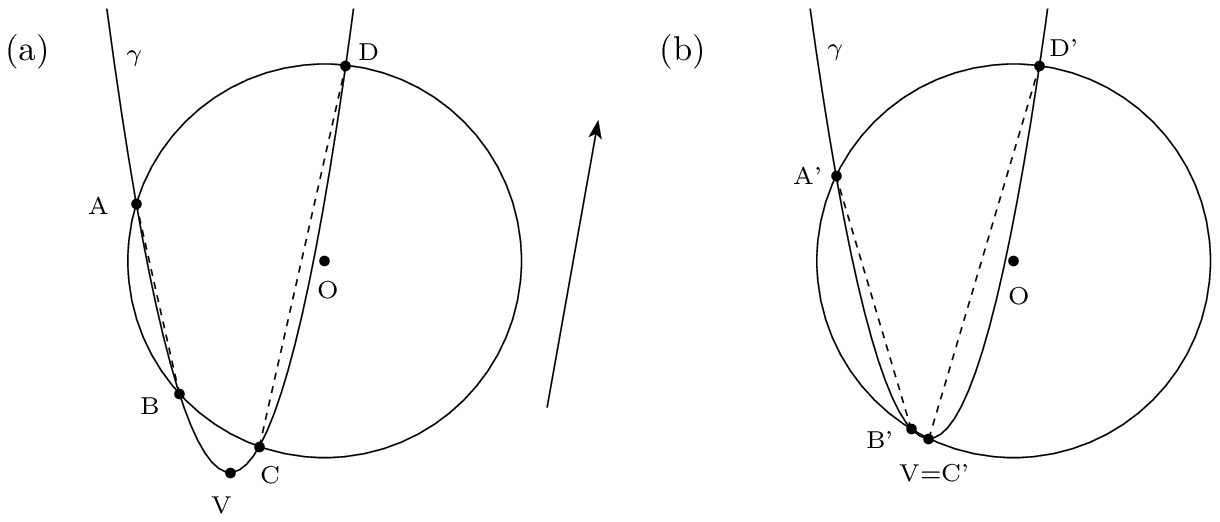}

\caption{}
   \label{ontheleft}
	
\end{center}
\end{figure}

In the first case, we have necessarily (see Fig. \ref{ontheleft} (a)):
\begin{eqnarray} \label{propriedades.pontos}
	y(D)>y(A)>0,  \ y(B), y(C)<0 \quad \mbox{and} \quad x(A)< x(B)< x(C)<0.
\end{eqnarray}
Taking a short translation of $\gamma,$ we get news points $A', B', C', D' \in \gamma \cap \partial B^2.$ To keep the properties (\ref{propriedades.pontos}) for the news points, we translate $\gamma$ such that $V \rightarrow C',$ and $|\overline{C'D'}|=|\overline{CD}|$ is constant during the translation. This is possible considering the following map: 
$$\mathcal{F}: C'=\gamma(x) \ \mapsto \ D' \in \gamma.$$
Where $x\geq x(V)$ and $D'$ is chosen such that $x(D')\geq x(C'),$ $|\overline{C'D'}|=|\overline{CD}|.$ Let's see that the above map is well defined with respect the choice of $D'.$ Indeed, as $\gamma$ is increasing for $x\geq x(V)$ and $|(x, \gamma(x))|\rightarrow \infty $ as $x \rightarrow \infty,$ the choice of $D'$ always exists and it is unique. This monotonicity also implies that $\mathcal{F}$ is one-to-one. Consider $\mbox{Dom}(\mathcal{F}):=\{\gamma(x): x\in [x(V), x(C)\}$ and $\mbox{Im}(\mathcal{F}):= \{\gamma(x): x\in [x(D_0),x(D)]\},$ where $D_0$ is the unique point in $\gamma$ such that $|\overline{VD_0}|=|\overline{CD}|$ and $x(D_0) \geq x(V).$ Again, the monotonicity of $\gamma$ implies that $\mathcal{F}$ is onto. Moreover, by the continuity of the distance function, we have that $\mathcal{F}: C' \mapsto D'$ is continuous. Take $E'$ the unique point such that $|\overline{E'C'}|=|\overline{E'D'}|=1$ and $y(E')\leq y,$ for all $y \in \overline{C'D'}.$ Consider the unity disk through the points $ C'$ and $D'$ with center at $E'.$ So we can take the inverse function of $\mathcal{F}$ and move continuously the unity disk such that $C' \rightarrow V,$ keeping $|\overline{C'D'}|=|\overline{CD}|.$ Note that $E'$ moves continuously to down and to left, since $\gamma$ is \-{in}\-{crea}\-{sing}. This is \-{equi}\-{va}\-{lent} to translate continuously $\gamma$ such that $V \rightarrow C',$ keeping $|\overline{C'D'}|=|\overline{CD}|,$ the curve gamma moves continuously to up and to right, and $\overline{C'D'}$ keeps on the left of $O.$ The last two implies that $y(A'), x(A'), x(B')$ increase, $y(B')$ decreases, and the properties (\ref{propriedades.pontos}) still hold during the move. Indeed, this is obvious for a short move and it holds during the translation because $\overline{C'D'}$ keeps on the left of $O,$ which implies that $y(A')<1, x(A'), x(B'), x(C'), y(B')<0.$ Also, as $x(A)'<x(B')<x(V),$ we have that $y(C')<0.$ By the triangle argument  and the fact that $\overline{C'D'}$ keeps on the left of $O,$ we see that $\gamma\big|_{C'}^{D'} \subset B^2$ during the translation. Also, $A', C', D'$ are distinct and $V \notin B^2$ during the move, thus there exists exact more one point $B' \in \gamma \cap \partial B^2.$  In particular, $\gamma\big|_{A'}^{B'} \subset B^2$ during the translation. Here we use the fact that $A'$ cannot be tangent to $\partial B^2,$ since $x(A')<0,$ $y(A')>0$ and the monotonicity properties of $\gamma.$

Note that $|\overline{A'B'}|$ is increasing, $B'$ is approaching to $V,$ and then $L(\gamma)\big|_{A'}^{B'}$ is increasing. We also note that $L(\gamma)\big|_{C'}^{D'}$ increases because $C'$ is approaching to $V,$ and then the curvature of $\gamma$ is increasing between $C'$ and $D'.$ In the end of the translation we get that $L(\gamma\cap B^2)$ increases, and $V\in B^2$ (see Fig. \ref{ontheleft} (b)).

In the second case, we just know that $x(A), y(B), y(C)<0$ and $y(A), y(D), $ $ x(D)>0$ (see Fig. \ref{no.meio} (a)). In this case, we take a short translation to up of $\gamma$ to get news points $A', B', C', D'$ with the same previous properties. We take this translation as long as $\overline{A'B'}$ and $\overline{C'D'}$ do not pass through the origin $O,$ and $V \notin \partial B^2.$ In particular, the last properties hold for the news points. Note that $y(A'), y(D')$ increase and $B', C'$ are approaching to $V,$ so we have that $|\overline{A'B'}|, |\overline{C'D'}|$ increase and therefore $L(\gamma)\big|_{A'}^{B'}$ and $L(\gamma)\big|_{C'}^{D'}$ are increasing, since the curvature of $\gamma \cap B^2$ is increasing. We stop this translation when $V$ touches $\partial B^2,$ or when $\overline{A'B'}$ or $\overline{C'D'}$ pass through the origin, in the last case we continue with the translation as in the previous case until $V$ touches $\partial B^2.$ In the end, we get again that $L(\gamma\cap B^2)$ increases and $V\in B^2$  (see Fig. \ref{no.meio} (b)).  


  \begin{figure}[ht]

\begin{center}

\includegraphics[trim=80 555 185 75,clip,scale=1]{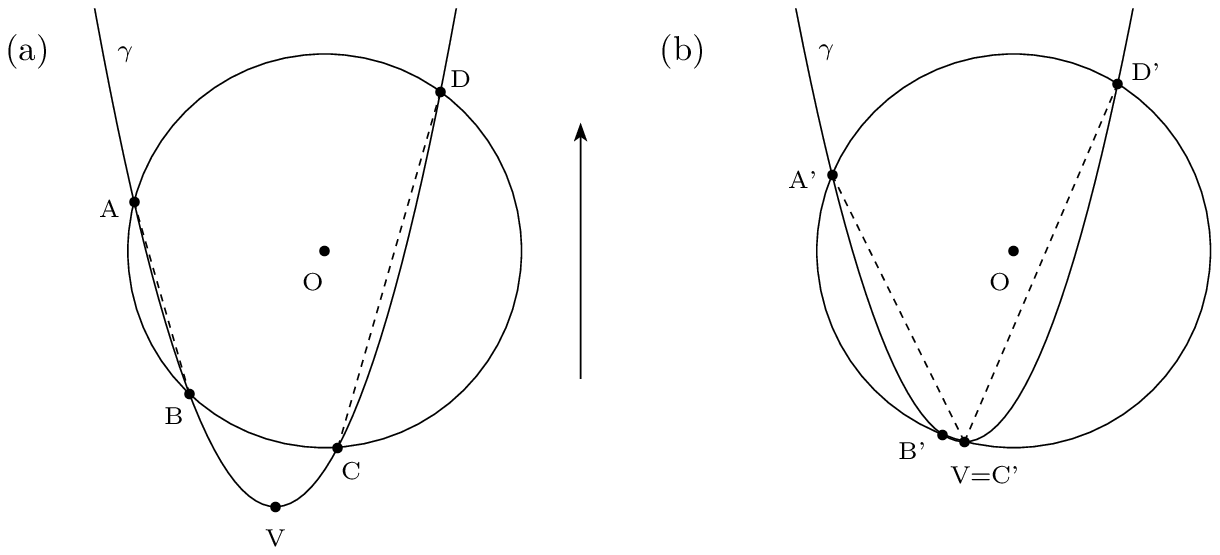}

\caption{}
  \label{no.meio}
	
\end{center}
\end{figure}

By the above arguments we simplify our analysis to the case  such that there exist two points $A, D \in \gamma \cap \partial B^2$ with $y(A), y(D)>0,$ and $V \in B^2.$ If we suppose that $x(A)<x(D),$ then by the triangle argument we have that $x(A)<x(V)<x(D).$
For a fix $\gamma$ in this case, the maximum length of $\gamma \cap B^2$ is reached when $x(V)=0,$ that is, when we translate $\gamma$ horizontally such that $x(V)\rightarrow 0.$ Indeed, suppose for now that $\gamma \cap B^2$ is connected and consider a such  translation so that $x(V)\rightarrow 0$ and $A \rightarrow A', D \rightarrow D'.$  If $x(D)<0,$ then during the translation and as long as $x(D')<0,$ we have that the curve goes inside $B^2,$ and in particular $L(\gamma\cap B^2)$ increases (see Fig. \ref{parabola.simetria} (b)). Now, consider the case $x(D)\geq 0$ and $x(A)<0$ as in the Fig. \ref{parabola.simetria} (a). In this case, let $|x(V)|=\epsilon>0.$ Take $E=(-(x(D)+\epsilon), y(D)), F=(-x(D), y(D)), G=(x(A)+\epsilon, y(A)), H=(x(A'),y(D)),$ and  $I=(x(A'),y(A)),$ we claim that $|\overline{EH}|<|\overline{IG}|.$ To see this, we take   the tangent line $r(x)$ to $\gamma(x)$ at $A',$ and we take $J=(r^{-1}(y(D)), y(D)), L=((r^{-1}(y(A)), y(A)).$ Note that $\overline{HI}$ is orthogonal to $\overline{JF}$ and to $\overline{AG},$ also $E, H \in \overline{JF},$ $I, L \in \overline{AG}$ and $\epsilon=|\overline{EF}|=|\overline{AG}|$ (see Fig. \ref{parabola.simetria} (a) and \ref{parabola.simetria.ampliada} (a)). As $|\overline{EF}|=|\overline{AG}|,$ if $|\overline{EH}|\geq|\overline{IG}|$ then we would have $|\overline{HF}| \leq |\overline{AI}|$ and, therefore, $|\overline{HA'}|<|\overline{A'I}|,$ because $A' \in B^2$ and $x(A')<0, y(A')>0.$ Using this and the fact that $\gamma$ is convex we see that $|\overline{EH}|<|\overline{JH}|<|\overline{IL}|<|\overline{IG}|,$ which is a contradiction. Let $\overline{\gamma}$ be the curve $\gamma$ after the translation. The inequality $|\overline{EH}|<|\overline{IG}|$ means that $L(\overline{\gamma})\big|_{A'}^{E}< L(\overline{\gamma})\big|_{A'}^{G},$ since the curvature of $\gamma$ is strictly increasing in the direction of the vertex $V.$ Thus, the length of $\gamma\cap B^2$ increases after the translation $A \rightarrow A',$ $D \rightarrow D', V \rightarrow V'$ because $L(\overline{\gamma})\big|_{E}^{A'}$ is the amount of the curve that goes outside $B^2,$ and $L(\overline{\gamma})\big|_{A'}^{G}$ is the amount that goes inside $B^2.$ Here, we are using the fact that $\overline{\gamma} \cap B^2$ has at most four points, the symmetric of $\overline{\gamma},$ $y(A')=y(D')>0$ and the fact that $V' \in B^2$ to conclude that $\overline{\gamma}\big|_{A'}^{V'} \subset B^2$ and $\overline{\gamma}\big|_{V'}^{D'} \subset B^2.$


\begin{figure}[ht]

\begin{center}

\includegraphics[trim=70 570 190 75,clip,scale=1]{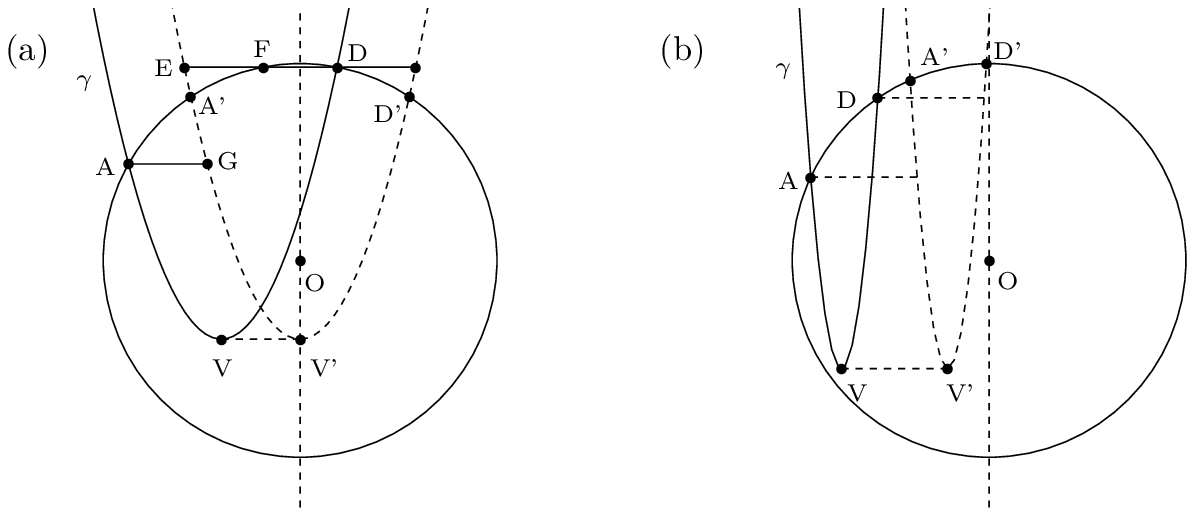}

\caption{}
  \label{parabola.simetria}
	
\end{center}
\end{figure}

\begin{figure}[ht]

\begin{center}

\includegraphics[trim=80 545 178 75,clip,scale=1]{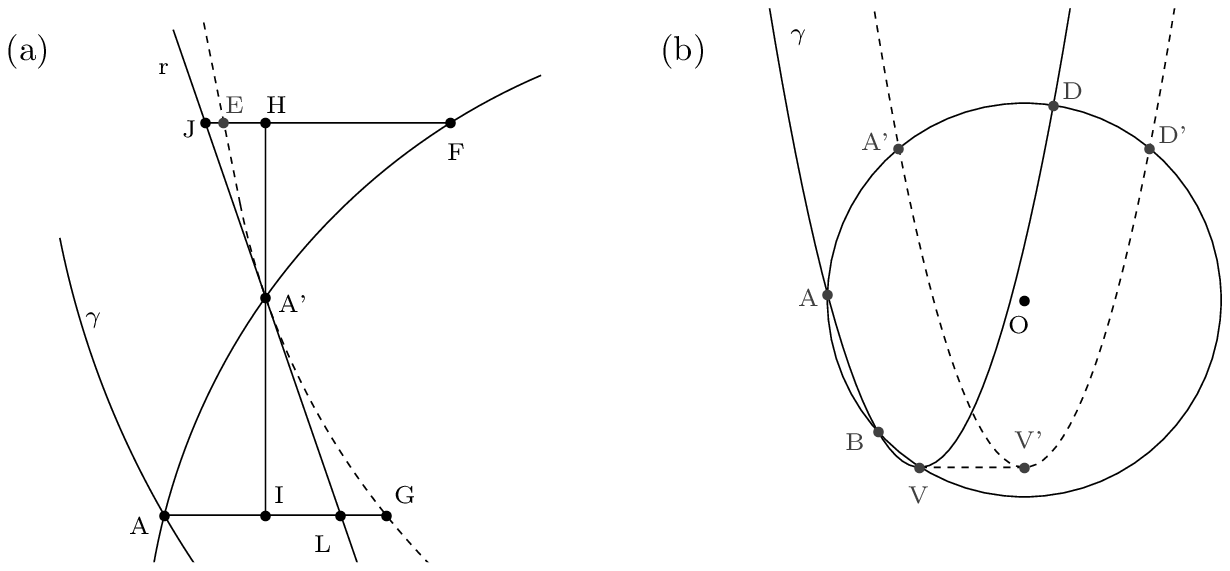}

\caption{}
  \label{parabola.simetria.ampliada}
	
\end{center}
\end{figure}

For the case $\gamma \cap B^2$ not connected, we have similarly that the length of $\gamma \cap B^2$ also increases after the above translation $x(V) \rightarrow 0.$ Indeed, as before we are supposing that there exist $A, D \in \partial B^2$ such that $y(A), y(D)>0$ and $V \in B^2.$ In particular, $x(A)<x(V)<x(D)$ if $x(A)<x(D).$ In this case it is enough suppose that there exist two connected components and the points $A, B, C, D \in \gamma \cap B^2$ such that $\gamma$ is inside $B^2$ between $A$ and $B,$ and between $C$ and $D;$ otherwise it is outside. Note that during the translation $y(A)$ increases. In the end, we get new points $A', D'$ such that  $y(D')=y(A')>0,$ and $V'\in B^2,$ since the vertex $V$ is the global minimum of $\gamma.$ Finally, as before $\overline{\gamma}\big|_{A'}^{V'} \subset B^2$ and $\overline{\gamma}\big|_{V'}^{D'} \subset B^2.$ In particular,  $L(\overline{\gamma} \cap B^2)$ increases, since the curve goes inside $B^2$ for $\overline{\gamma}(x)< y(A),$ and the previous paragraph for $\overline{\gamma}(x)\geq y(A)$ (see Fig. \ref{parabola.simetria.ampliada} (b).)

By the last two paragraphs, we need to find an upper bound for $L(\gamma \cap B^2),$ when $x(V)=0,$  $\gamma \cap B^2$ is connected, and $\{\gamma \cap B^2\} \backslash \{V\}$ is given by two points $A,D$ such that $-x(A)=x(D),$ and $y(A)=y(D)>0.$ In this situation, if $y(V)>-1,$ we can translate $\gamma$ to down such that $y(V)\rightarrow -1,$ then $L(\gamma \cap B^2)$ increases as long as $y(A')=y(D')>0.$ So, we consider the last hypothesis above with $V'=(0,-1),$ in other words, the curve is tangent to $\partial B^2$ at $V'$ (see Fig. \ref{parabola.final} (a)). Here we are using that the length $L(\gamma \cap B^2),$ for $V'=(0,-1)$ and $\gamma$ passes through $A', D'$ is bigger than the length of $L(\widetilde{\gamma} \cap B^2),$ if $\widetilde{\gamma}$ passes through $A', D'$ and $V \in \mbox{int}(B^2).$


\begin{figure}[ht]

\begin{center}

\includegraphics[trim=80 573 177 75,clip,scale=1]{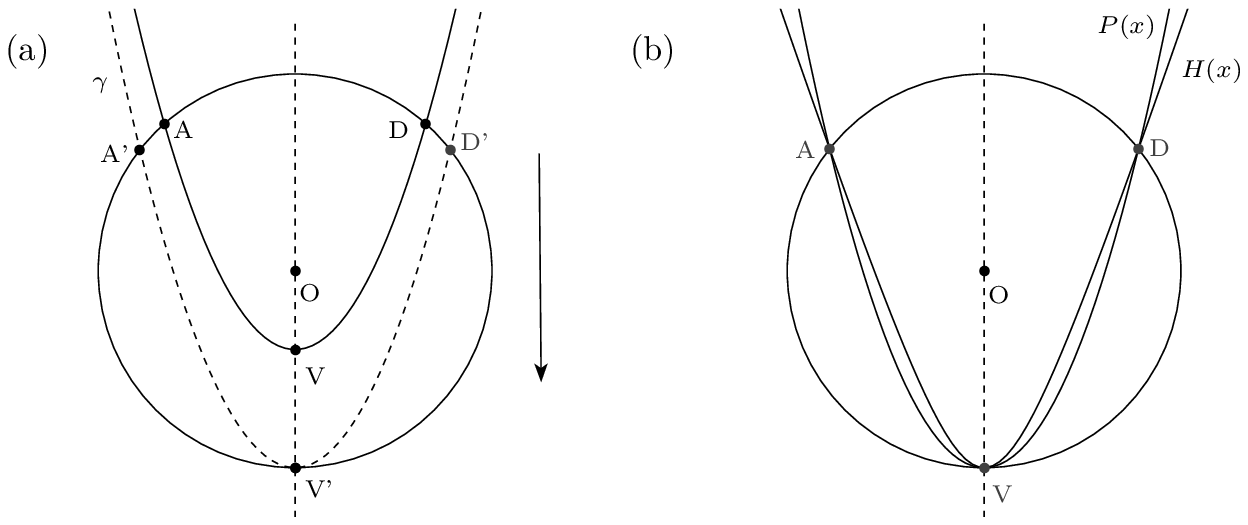}

\caption{}
 	\label{parabola.final}
	
\end{center}
\end{figure}

Remember, the curve $\gamma(x)$ can be a hyperbola $H(x)$ with a unique branch intersecting $B^2,$ or a parabola $P(x).$ To satisfy our situation, the equations become, respectively 
$$H(x)=\frac{c}{d}\sqrt{d^2+x^2}-(1+c)   \quad \mbox{and} \quad P(x)=ax^2-1,$$
where $a, c, d >0,$ $H(1)=(c/d)\sqrt{d^2+1}-(1+c)>0,$ and $P(1)=a-1>0.$ Thus, $H(x) \cap B^2 \cap \mathcal{H}^+$ and $P(x) \cap B^2 \cap \mathcal{H}^+$ have two points, also $H(x)$ and $P(x)$ pass through $V=(0,-1).$ Suppose that $H(x)$ and $P(x)$ pass through the same points $A, D \in B^2 \cap \mathcal{H}^+,$ then $H(x) \cap P(x)=\{A, D, V\}.$ This is because of the symmetry, and the fact that the intersection $H(x) \cap P(x)$ has at most four points.  Note that $H(x)$ and $P(x)$ are tangent at $V,$ so $A$ (and then $D,$ by symmetry) cannot be a point where $H(x)$ and $P(x)$ has a common tangent, since two distinct conics are tangent at most at two points. As $P(x)>H(x)$ for $x \rightarrow \infty,$ we conclude that the graphic of $H(x)$ is above of $P(x)$ for $x(A)\leq x \leq x(D).$ In particular, $L(H(x)\cap B^2)<L(P(x)\cap B^2),$ so we only need to bound $L(P(x)\cap B^2)$ for $a>1 $ (see Fig. \ref{parabola.final} (b)). As $a>1,$ the points $A, D$ can be determined uniquely by the value of the parameter $a.$ In fact, $-x(A)=x(D)=x(a)=\sqrt{2a-1}/a,$ where $x(a)$ is the positive solution of $x^2+(ax^2-1)^2=1.$ Then, we can calculate $L(P(x) \cap B^2)$ in the parameter $a:$

\begin{eqnarray*}
	L(P(x) \cap B^2)&=&L(a) \\
	&=& 2 \int_{0}^{x(a)} \sqrt{1+4a^2 x^2} \ dx \\
	&=&\frac{\ln\left(\sqrt{4a^2 x(a)^2+1}+2ax(a)\right)+2ax(a)\sqrt{4a^2x(a)^2+1} }{2a} \\
	&=&\frac{\ln\left(\sqrt{8a-3}+2\sqrt{2a-1}\right)+2\sqrt{2a-1}\sqrt{8a-3}}{2a}.
\end{eqnarray*}

By the expression above we have that $L(P(x) \cap B^2) \rightarrow 4$ as $a \rightarrow \infty$ ($P(x) \cap B^2$ becomes two diameters). We see below that $L(a)$ has a global maximum point at $a_0<\infty,$ and then $L(a_0)>4.$ Indeed, 
\begin{eqnarray*}
L'(a)=\frac{8a-3-(1/2)\ln\left(\sqrt{8a-3}+2\sqrt{2a-1}\right)\sqrt{8a-3}\sqrt{2a-1}}{a^2 \sqrt{2a-1}\sqrt{8a-3}}.
\end{eqnarray*}
Taking $z=2a-1,$ the denominator above becomes
\begin{eqnarray*}
	 4z+1-\frac{1}{2}\ln\left(\sqrt{4z+1}+2\sqrt{z}\right)\sqrt{4z+1}\sqrt{z}.
\end{eqnarray*}
So, the sign of $L'(a)$ is the sign of 
\begin{eqnarray} \label{sign.derivative.lenght}
	 \displaystyle \frac{2\sqrt{4z+1}}{\sqrt{z}}-\ln\left(\sqrt{4z+1}+2\sqrt{z}\right).
\end{eqnarray}
Note that the expression above starts positive for $a>1$ and tends to $-\infty$ when $a\rightarrow \infty,$ moreover it is strictly decreasing for $a>1/2.$ The latter is because the derivative of the last expression is given by
\begin{eqnarray*}
	-\frac{z+z^2}{z^{5/2}\sqrt{4z+1}}.
\end{eqnarray*}  
Therefore, there exists a unique $a_0>1$ such that $L'(a_0)=0.$ Moreover,  
$L(a)$ is strictly increasing for $1<a<a_0,$ and it is strictly decreasing for $a>a_0.$ In particular, $L(a_0)>4$ and $L(a_0)$ is the global maximum of $L(a),$ since $L(a) \rightarrow 4$ as $a \rightarrow \infty.$ We can estimate $a_0$ such that (\ref{sign.derivative.lenght}) becomes zero, and we obtain $a_0 \approx 94.091282,$ and then $L_1<L_0=L(a_0)\approx 4.00267.$  \hfill $\Box$

\vspace{0.4cm}

\bibliographystyle{plain}
\bibliography{sample}

\vspace{0.8cm}

\end{document}